\documentclass{article}
\usepackage{amsmath}
\usepackage{amssymb}
\usepackage{amsthm}
\usepackage[latin1]{inputenc}
\usepackage[T1]{fontenc}
\usepackage{lmodern}
\usepackage{fullpage}
\usepackage[english]{babel}
\usepackage{amsfonts}
\usepackage{graphicx}
\usepackage{verbatim}
\usepackage{color}
\usepackage{stmaryrd}
\graphicspath{{Figures/}}

\definecolor{mybrown}{RGB}{148,0,211}

\definecolor{liens}{rgb}{0.,0.,0.8}
\usepackage[colorlinks,
            linkcolor=liens,
            anchorcolor=liens,
            citecolor=liens,
            filecolor=liens,
            menucolor=liens,
            urlcolor=liens,
            bookmarks,
            bookmarksnumbered,
            pdfstartview=FitH]{hyperref}
\usepackage{multirow}
\usepackage[symbol]{footmisc}

\newcommand{\RR}{\mathbb R}
\newcommand{\NN}{\mathbb N}
\newcommand{\DD}{\mathbb D}
\newcommand{\CC}{\mathbb C}
\newcommand{\TT}{\mathbb T}

\newcommand{\ZZ}{\mathbb Z}

\newcommand{\cP}{{\cal P}}
\newcommand{\cR}{{\cal R}}
\newcommand{\ee}{e}
\newcommand{\abs}[1]{\ensuremath{|#1|}}

\newcommand{\Abs}[1]{\ensuremath{\left|#1\right|}}

\newtheorem{theorem}{Theorem}

\newtheorem{prop}{Proposition}
\newtheorem{lem}{Lemma}
\newtheorem{remark}{Remark}

\begin{document}

\title{Lower bounds in $H^2$-rational approximation to Blaschke products}
\author{Laurent Baratchart\footnotemark[1] \and
  Alexander Borichev\footnotemark[2] \and
  Sylvain Chevillard\footnotemark[1] \and
  Claire Coiffard\footnotemark[3] \and
  Rachid Zarouf\footnotemark[3]${}^{~,}$\footnotemark[4]\\[0.3cm]
\url{laurent.baratchart@inria.fr}\\
\url{alexander.borichev@math.cnrs.fr}\\
\url{sylvain.chevillard@inria.fr}\\
\url{claire.coiffard-marre@univ-amu.fr}\\
\url{rachid.zarouf@univ-amu.fr}
}
\date{September 2026}

\maketitle

   \footnotetext[1]{Universit\'e C\^ote d'Azur, Inria, 2004 route des Lucioles, BP 93, 06\,902 Sophia Antipolis Cedex, France.}
   \footnotetext[2]{Aix-Marseille Universit\'e, CNRS, I2M, Marseille, France.}
   \footnotetext[3]{Aix-Marseille Universit\'e, Laboratoire ADEF, Campus Universitaire de Saint-J\'er\^ome, 52 avenue Escadrille Normandie Ni\'emen, 13\,013 Marseille, France.}
   \footnotetext[4]{CPT, Aix-Marseille Universit\'e, Universit\'e de Toulon, 13\,288 Marseille, France.}

\paragraph{Abstract:}
We derive  lower bounds in best rational approximation of given degree 
to finite Blaschke products, in the Hardy space $H^2$ of the unit disk.
We first consider approximation to $z^N$, and then move on to more
general Blaschke products  whose zeros are bounded away from the circle.
The latter case depends on Fourier coefficients estimates for Blaschke products
which are of independent interest.


\paragraph{Acknowledgment:} the first and third authors thank Alexandre Fran\c{c}ois for bringing the question addressed in Theorem \ref{initp} to their attention.

\paragraph{Data Availability Statement:} the results presented in this
article do no need any external data to get reproduced. The numerical
experiments have been performed using the RARL2 software (see
\url{https://project.inria.fr/rarl2}).

\newpage
\section*{Foreword}
This document is the author version of an article published in \emph{Research in the Mathematical Sciences}. The following slight improvements have been implemented in the present document.
\begin{itemize}
\item The statement of Theorem~\ref{lemH1} was not exactly what we need when we use it in the proof of Theorem~\ref{genBp}, and neither was it what is actually showed in the proof of Theorem~\ref{lemH1}. We corrected the statement (the min is taken over those $b_n \in B_n$ whose degree is exactly $n$ rather than over all $b_n \in B_n$) and slightly precised the proof. Notice that the alternative statement also holds, as is now explained in a remark following the proof.
\item After fixing a minor bug in RARL2 and running it again on the examples showed in Figures~\ref{firstBound} and~\ref{secondBound}, marginally better approximations have been obtained for some degrees. The figures have been updated accordingly.
\item In the appendix, the notation $\widehat{B}(k)$ for the Taylor-Fourier coefficients of $B$ is now introduced before it is used in the statement of Proposition~\ref{estFTnpa}. The remark following Proposition~\ref{estFTnpa} is now formulated in a way that makes a clearer parallel with the result presented in the proposition.
\end{itemize}

\tableofcontents

\newpage
\section{Introduction}
Rational approximation to analytic functions of a single complex variable
is an old,  classical piece of mathematical analysis  with famous applications to
number theory and spectral theory; see for example \cite{Siegel,KR,Nikolskii, GoNoHe}
and their bibliography. It has become  an efficient numerical tool and a
cornerstone
of modeling and design in applied mathematics and engineering.
Let us mention  \cite{Gr65,GoTre,GuSem,TeToGa, OSM, Pozzi, AvFaRe,YuWang,Tj_PRA77,BMSW06,BLS2007} for a sample.
Special interest attaches to the case where one approximates
on a closed curve a function that
extends holomorphically on one side of that curve.
In connection with deconvolution, system identification and control,
such issues typically arise on the line or the circle where they make 
contact  with extremal problems in
Hardy spaces \cite{B_CMFT99,GL84,Peller,Nikolskii2,OSM,IIS,YuWang}. 
Our model curve in this paper is the unit circle, though things
translate easily to the line, and approximation will take place in the least squares sense.

From a theoretical point of view, a lot of work has been devoted to
error rates and how they reflect  smoothness properties of the approximated function.
We mention converse theorems by Peller
on the speed of rational approximation
\cite{Peller}, the construction by Glover of near-best uniform rational
approximants \cite{GL84}, Parfenov's solution of 
the Gonchar conjecture  on the degree of rational approximation 
on compact sets of the domain of analyticity
\cite{Par86}, the Gonchar-Rakhmanov $n$-th root estimates in
uniform rational approximation {\it via} interpolation to 
sectionally holomorphic functions about an $S$-contour  and their
generalizations  to best $L^2$ and $L^p$ approximants  in
\cite{BaYaSta,Totik,BaStaYat}.

From a constructive point of view, finding demonstrably
optimal or near-optimal
approximants is extremely difficult except in special cases, and usually
requires more knowledge on the function than may possibly be obtained.
As a consequence, attention has focused mainly on computationally attractive
schemes dwelling on interpolation, like Pad\'e interpolants and
variants thereof \cite{BakerGravesMorris,St89,Nu2,Lub,Lubinsky1,AAA,SilYuTreBe}.

One fairly open issue, though, is the derivation of lower bounds on the
achievable approximation error in given degree. Such bounds are key to assess the quality of an approximant and play an important role in questions of
uniqueness,
see \cite{BSW96,BStW01b,BStW99,BS02,BaYaunic}. An early reference seems to be
\cite{Gloverbound}, where a lower bound on the achievable error in
$L^2$ rational approximation of degree $n$ to a Hardy function whose
Hankel operator is Hilbert-Schmidt (by \cite[Ch. 6, Theorem 2.1]{Peller} this amounts to require
that this function is of Besov $B_2^{1/2}$-class on the circle) is obtained
in terms of the singular values of a weighted version of that operator.
Besides \cite{Gloverbound},  most results on this topic
that we are aware of can be found in \cite{BCQ}, of which
the present paper is, in a sense, a sequel.
In the last reference, lower bounds in $L^2$ rational approximation  of given degree to Hardy functions in $H^2$ are obtained {\it via} a comparison
with AAK (meromorphic) approximants, using   topological tools.

However, \cite{BCQ} does not offer lower bounds depending explicitly
on the degrees of the approximant and of the approximated function, assuming that the latter is rational. Below, we derive such bounds in the special case where one approximates
a unimodular rational function; namely, a finite Blaschke product.
This case is particularly difficult in that uniform approximation is not possible (if $f$ is a finite Blaschke product of degree~$N$, the unique best rational approximant with respect to the uniform norm of degree no larger than $n < N$  is the zero function, yielding an error of norm~$1$; and when $n \ge N$, the function $f$ itself is an admissible approximant, which yields  an error of norm~$0$). So one can surmise that $L^2$ approximation should be slow. In the simplest case
where the approximated function is a power of the complex variable, we obtain
a very explicit lower bound which is not too far off, judging from our numerical experiments using the RARL2 software; in this case at least, it compares favorably  to the bound
from \cite{Gloverbound} which, in its author's own words, ``is unlikely to be achievable since it ignores the constraint that'' [the approximating lower rank operator]
``should be a Hankel matrix''.  Such a demonstrably slowly decreasing lower bound,  when the degree of the approximant increases, accounts for the experimental fact that, in the language of system theory, a pure delay system hardly lends itself to model reduction; or, in more trendy terminology, a linear neural net cannot well reproduce
delays. 
The situation with general Blaschke products is more complex and the bound
less convincing, but still our result
seems to be the first of this kind and the proof leads us to derive
estimates for Fourier coefficients of Blaschke products which are of independent interest. 

We emphasize again
that the $L^2$ norm and
weighted variants thereof are of great importance in applications, due to 
their interpretation as a variance in a stochastic context. 
Moreover, $L^2$-best rational approximants in Hardy space have the 
remarkable property of being particular multipoint Pad\'e interpolants (the
interpolation points are of course unknown and depend on the function) \cite{Levin}. 

The paper is organized as follows. After setting up notation and recalling
some basic  degree theory and the Br\'ezis formula in Section \ref{prelim},
we discuss Hardy spaces in Section \ref{Hardysec} and
best rational (or meromorphic) approximation on the circle in
Section \ref{Rat}.  We present lower bounds in rational approximation to
$z^N$  and then to more general Blaschke products in Section \ref{Lbounds},
before concluding with a couple of numerical experiments using the RARL2 software in Section~\ref{sec:numericalresults}.

\section{Notations and Preliminaries}
\label{prelim}
Let $\DD$ be the unit disk and $\TT$ the unit circle in the complex plane $\CC$. We put $z=x+iy$ to denote the real and imaginary parts of $z\in\CC$, which makes for the identification $\CC\sim \RR^2$.

We write $\cP_n[z]$ for the space of complex algebraic polynomials
of degree at most $n$ in the variable $z$, or simply $\cP_n$ if the
variable is understood. Below, we let
$\mathcal{Z}(q)$ indicate the set of zeros of a polynomial $q$. 
For $q_n\in\cP_n[z]$, we define its
{\emph{reciprocal polynomial}} to be
\begin{equation}
  \label{defrec}
  \widetilde{q}_n(z):=z^{n}\,\overline{q_n(1/\bar z)}.
  \end{equation}
One should be careful that this definition depends on $n$: if we consider
$q\in\cP_{n-1}$ as an element of $\cP_n$ with zero leading 
coefficient, the two definitions of $\widetilde{q}(z)$ in $\cP_{n-1}$ and
$\cP_{n}$ may be inconsistent. Therefore we always specify
which definition is used; {\it e.g.},
via a subscript ``$n$'' for $q_n$ like we did in \eqref{defrec}. Clearly the ``tilde''
operation is an involution of $\cP_n$ preserving the pointwise modulus on $\TT$.

We designate by $\cR_{m,n}=\cR_{m,n}(z)$ 
the set of complex rational functions of type $(m,n)$ in
$L^2$, namely those that can be written as $p_m/q_n$
where $p_m$ belongs to $\cP_m$ and $q_n\in \cP_n$
has no roots on ${\TT}$. 
When $r=p_m/q_n$  is in irreducible form,
the integer $\max\{m,n\}$ is the (exact) degree of~$r$.

We denote by $C(\TT)$ the space of continuous, 
complex-valued functions on $\TT$, and by $C^\infty(\TT)$ the subspace of smooth functions; that is, functions $f(e^{i\theta})$ that are locally $C^\infty$-smooth functions of $\theta\in\RR/(2\pi\ZZ)$. Hereafter, for a function on $\TT$, a superscript ``prime'' always
refers to the derivative with respect to $\theta$.
Note that $f^\prime$ is indeed
well-defined, because
$\frac{d}{d\theta}f(e^{i\theta})$ is independent of the continuous
branch $\theta$ of the argument  which is chosen modulo $2\pi$.
For $1\leq p\leq \infty$,
we let $L^p=L^p(\TT)$ indicate the familiar
Lebesgue space of (equivalence classes of a.e. coinciding) complex measurable functions on $\TT$ 
such that
\[\|f\|_{p}=\left(\frac{1}{2\pi}
\int_{0}^{2\pi}|f(e^{i\theta})|^p\,d\theta\right)^{1/p}<\infty
\quad \text{if }1\leq p<\infty, \qquad
\|f\|_\infty={\rm ess.}\sup_{\theta \in 
  [0,2\pi]}|f(e^{i\theta})|<\infty.\]
The definition extends to $\CC^n$-valued functions, replacing moduli by
Euclidean norms.
In a similar vein, we put $C(\DD)$ and $C^\infty(\DD)$ for  continuous and smooth functions on $\DD$. We denote by $C(\overline{\DD})$ (resp. $C^\infty(\overline{\DD})$) those extending continuously to $\overline{\DD}$ (resp. smoothly
to a neighborhood of $\overline{\DD}$).
The Lebesgue spaces $L^p(\DD)$ are defined analogously, namely:
\[\|f\|_{L^p(\DD)}=\left(\frac{1}{\pi}
\int_{\DD}|f(z)|^p\,dm(z)\right)^{1/p}<\infty
\quad \text{if }1\leq p<\infty, \qquad
\|f\|_{L^\infty(\DD)}={\rm ess.}\sup_{z \in 
  \DD}|f(z)|<\infty,\]
where $m$ indicates Lebesgue measure.

The Sobolev space $W^{1,2}$ consists of complex functions on $\TT$,
absolutely continuous with respect to arclength, whose derivative 
$f^\prime (e^{i\theta}):=\frac{d}{d\theta}f(e^{i\theta})$ again lies in $L^2$.
It is a Banach space with  norm 
\begin{align}
\label{normST}
 \|f\|^2_{W^{1,2}}= 
\|f\|^2_{L^2}+\|f^\prime\|^2_{L^2}.
\nonumber
\end{align}
If one associates with $f:\TT\to\CC$ its Fourier coefficients $(a_k)_{k\in\ZZ}$, so that
$f(e^{i\theta})=\Sigma_{k = -\infty}^\infty a_ke^{ik\theta}$ is the Fourier expansion of $f$, one may identify $L^2$ with $l^2(\ZZ)$ and $W^{1,2}$ with the weighted
space $l^2(\ZZ,w)$, where the weight is $w(k)=1+|k|^2$; moreover this identification is isometric, because by Parseval's relation one has
\[\|f\|_2=\left(\sum_{k=-\infty}^{\infty} |a_k|^2\right)^{1/2}\quad \mathrm{ and}\quad
  \|f^\prime\|_2=\left(\sum_{k=-\infty}^{\infty} |k|^2|a_k|^2\right)^{1/2}.\]
The fractional Sobolev space $W^{1/2,2}$ is a real interpolation space of exponent 
$1/2$ between $L^2$ and $W^{1,2}$,
a norm on which is given in terms of the Fourier coefficients (see \cite[thm 7.55]{Adams}) by
\begin{equation}
  \label{norm123}
  \|f\|_{1/2}^2=|a_0|^2+\sum_{k=-\infty}^\infty |k||a_k|^2, \qquad
  f(e^{i\theta})=\sum_{k=-\infty}^\infty a_ke^{ik\theta}.
  \end{equation}
Though the norm in \eqref{norm123} is the one we shall use, 
let us mention that an equivalent  norm on $W^{1/2,2}$ is given in terms of the values of $f$ by
\begin{equation}
\label{defWt}
\|f\|_{W^{1/2,2}}:=\|f\|_{L^2}+
\left(\int_{\TT}\int_{\TT}
\frac{|f(e^{i\theta})-f(e^{i\theta_1})|^2}{\left|e^{i\theta}-e^{i\theta_1}\right|^2}\,d\theta d\theta_1
\right)^{1/2},
\end{equation}
see \cite[{ thm} 7.47 \& { rem. 7.45} ]{Adams}.
In fact, the norm in \eqref{defWt} characterizes $W^{1/2,2}$ as the space of traces on $\TT$ of Sobolev functions in $W^{1,2}(\DD)$; the latter space
consists of complex functions
in $L^2(\DD)$ whose distributional derivatives of the first order again lie in
$L^2(\DD)$. We refer the reader to \cite{Adams} for more on Sobolev spaces.

The winding number $W(f)$ of a nowhere vanishing $f\in C(\TT)$ is the number of times the tip of the vector $f(e^{i\theta})$ winds around the origin as $\theta$ traverses $[0,2\pi)$. When $f\in C^\infty(\TT)$, so that $f=\rho e^{i\alpha}$ for some smooth functions $\rho:\TT\to\RR^+$ and $\alpha:\TT\to\RR/(2\pi\ZZ)$
with  $\rho>0$, the winding number  may  formally be defined as
\begin{equation}
  \label{defind}
  W(f):=\frac{1}{2\pi}\int_0^{2\pi}\left(\alpha(e^{i\theta})\right)^\prime d\theta =
  \frac{1}{2i\pi}\int_0^{2\pi}\frac{(e^{i\alpha})^\prime }{e^{i\alpha}}(e^{i\theta})d\theta=\frac{1}{2i\pi}
  \int_0^{2\pi}\frac{f^\prime}{f}(e^{i\theta})d\theta,
\end{equation}
where we used that
\[\int_0^{2\pi}\frac{(\rho(e^{i\theta}))^\prime }{\rho(e^{i\theta})}d\theta=\Big[ \log\rho(e^{i\theta}) \Big]_{\theta=0}^{\theta=2\pi}=0.
  \]
  When $f$ is merely continuous and never vanishing, for each $\varepsilon>0$ one can pick a smooth
  function $g\in C^\infty(\TT)$ such that $|f-g|<\varepsilon |f|$ on $\TT$, and
  for $\varepsilon$ small enough any two such $g$ are smoothly homotopic in a never vanishing manner, hence
  they have the same winding number because the latter is integer-valued and varies continuously along the homotopy. This common winding number is, by definition, the one of $f$. In particular, if $f\in C(\overline{\DD})$ is never zero
  on $\TT$ and meromorphic in $\DD$, then
  \begin{equation}
    \label{indmero}
    W(f_{|\TT})=\sharp Z(f)-\sharp P (f),
  \end{equation}
  where $\sharp Z(f)$ is the number of zeros and $\sharp P(f)$ the number of poles of $f$ in $\DD$, counting multiplicities. Indeed, $f$ can be uniformly approximated by $f_r(e^{i\theta}):=f(re^{i\theta})$ for $r$ small enough, and
  then \eqref{indmero} follows from \eqref{defind} and the residue formula.
  
  Let us stress that the winding number is a special case of the more general
  notion of topological
  degree, that applies in higher dimension as well. More precisely, for $F:\mathbb{S}^n\to\mathbb{S}^n$ a $C^\infty$-smooth map from the unit sphere of dimension $n$ into itself (viewed here as a smooth oriented manifold), the topological degree of $F$ is
  \begin{equation}
    \label{topd}
    \mathrm{deg}\,F:=\sum_{\xi\in F^{-1}(c)}\mathrm{{\bf orient}}\left(DF(\xi)\right),
    \end{equation}
  where $c$ is any regular value of $F$ (a value is regular if it is
  not attained at a point $\xi$ where the rank of the derivative $DF(\xi)$ drops) and $\mathrm{{\bf orient}}(DF(\xi))$  is equal to +1 if $DF(\xi)$ preserves orientation and to -1 if it
  reverses orientation. Note that the right hand side of \eqref{topd} is a well-defined finite sum because $c$ is a regular value. The topological degree of a continuous map $F$ is then defined as the one of a smooth map arbitrary close to $F$.
  When $n=1$ and $\mathbb{S}^1\sim \TT$ is oriented counterclockwise while $F$ is smooth so that $F=e^{i\alpha}$ for some smooth $\alpha:\TT\to \RR/(2\pi\ZZ)$, a regular value is a direction $e^{it_0}$
  which is attained under $F$ only at points $e^{i\theta}$ where
  $(\alpha(e^{i\theta}))^\prime\neq0$, and $\mathrm{{\bf orient}} (DF(e^{i\theta}))$ is $+1$ if  $F(e^{i\theta})$  traverses $e^{i t_0}$ counterclockwise ($(\alpha(e^{i\theta}))^\prime>0$)   and $-1$ if   it traverses clockwise ($(\alpha(e^{i\theta}))^\prime<0$). It is thus apparent that
  $\mathrm{deg}\, F$ algebraically counts  the number of times $F(e^{i\theta})$ winds around
  zero as  $\theta$ traverses $[0,2\pi)$.   Hence, the winding number of
  a continuous never-vanishing $f:\TT\to\CC$ is just the  topological degree
  of $\frac{f}{|f|}:\TT\to\TT$.  
  We refer the reader to \cite[Chapter 3]{GP} for basic facts on the topological degree (see in particular the second corollary on p. 115 of that reference), and  its computation as a sum of
  local indices at the singular points of any  nondegenerate extension
  of $F$ inside the ball that generalizes  \eqref{indmero}
  (see proposition on page 108 and the computation in the proof
  of proposition on p. 110 of \cite{GP}).

  When $f:\TT\to\TT$ is continuous and lies in $W^{1/2,2}$ (membership in $W^{1/2,2}$ alone does not imply continuity), a remarkable formula of Br\'ezis asserts that
  \begin{equation}
    \label{degreeW12}
W(f)=\sum_{k\in\ZZ} k |a_k|^2,\qquad  f(e^{i\theta})=\sum_{k=-\infty}^\infty a_ke^{ik\theta};
    \end{equation}  
    see \cite{BreNi} for details.
    \section{Hardy spaces of the disk}
    \label{Hardysec}
 We let
$H^2=H^2(\DD)$ be the Hardy space of holomorphic functions in $\DD$ whose 
Taylor coefficients at 0 are square summable:
\[H^2=\{f(z)=\Sigma_{k=0}^\infty a_kz^k:\quad
\|f\|_{H^2}:=\left(\Sigma_{k=0}^\infty |a_k|^2\right)^{1/2}<+\infty\}.
\]
We refer the reader to \cite{Garnett} for standard facts on Hardy spaces.
By Parseval's relation 
\begin{equation}
\label{defintH2}\|f\|_{H^2}^2=\sup_{0\leq r<1}\frac{1}{2\pi}\int_0^{2\pi} 
|f(re^{i\theta})|^2\,d\theta,
\end{equation}
and  the map
\[\Bigl(f(z)=\Sigma_{k=0}^\infty a_kz^k\Bigr)\longrightarrow
\Bigl({f^*}(e^{i\theta}):=\Sigma_{k=0}^\infty a_ke^{ik\theta}\Bigr)
\]
is an isometry from $H^2$ onto the closed subspace of $L^2$
comprised of functions whose Fourier coefficients of strictly negative
index do vanish. As is customary, we 
identify $H^2$ with this subspace so that the 
distinction  between $f$ and ${f^*}$ as well as  
$\|f\|_{H^2}$ and $\|f^*\|_2$ will disappear.
This conveniently allows one to regard members of
the Hardy class both as functions on $\DD$ and on $\TT$.
From the function-theoretic viewpoint, the correspondence $f\mapsto f^*$
is that ${f^*}(e^{i\theta})$ is almost everywhere the limit of 
$f(z)$ as $z$ tends non-tangentially to $e^{i\theta}$ within $\DD$.

We put $\bar{H}^{2,0}=\bar{H}^{2,0}(\CC\setminus\overline{\DD})$ 
for the companion Hardy space of holomorphic functions in 
$\CC\setminus\overline{\DD}$,  vanishing at infinity, whose 
Taylor coefficients there  are square summable:
\[\bar{H}^{2,0}=\{f(z)=\Sigma_{k=1}^\infty a_kz^{-k}:\quad
\|f\|_{\bar{H}^{2,0}}:= \left(\Sigma_{k=1}^\infty |a_k|^2\right)^{1/2} <+\infty\}.
\]
The map
\[\Bigl(f(z)=\Sigma_{k=1}^\infty a_kz^{-k}\Bigr)\longrightarrow
\Bigl({f^*}(e^{i\theta})=\Sigma_{k=1}^\infty a_ke^{-ik\theta}\Bigr)
\]
is an isometry from $\bar{H}^{2,0}$ onto the closed subspace of $L^2$
comprised of functions whose Fourier coefficients of non-negative index do vanish, and as before we identify $\bar{H}^{2,0}$ with the latter.
We have the orthogonal decomposition:
\begin{equation}
\label{decompP}
L^2=H^2\oplus\bar{H}^{2,0}.
\end{equation}

Clearly, $f$ belongs to $\bar{H}^{2,0}$ if and only if the function
$\check f$ given by
\begin{equation}
\label{defcheck}
\check{f}(z):=z^{-1}\overline{f(1/\bar{z})}
\end{equation}
 lies in $H^2$,
and the map  $f\mapsto \check{f}$ is an involutive
isometry of $L^2$ sending $H^2$ onto $\bar{H}^{2,0}$. Actually
$|\check{f}|=|f|$ \emph{pointwise} on $\TT$, since  
$\overline{f(1/\bar{z})}=\overline{f(z)}$ when $|z|=1$. If $f$ 
is holomorphic on $\Omega$, then $f^\sharp(z)= 
\overline{f(1/\bar{z})}$ is holomorphic on the reflection of $\Omega$
across $\TT$, and if $f$ is rational $f^\sharp$ is likewise rational of the same degree.
By inspection, we get that $f^\sharp_{|\TT}=\bar{f}_{|\TT}$ where the subscript ${|\TT}$ indicates ``restriction on $\TT$''.
We let
\begin{equation}
  \label{defRiesz}
  {\bf P}_+\bigl(\Sigma_{k\in\ZZ}\,\, a_ke^{ik\theta}\bigr)=
\Sigma_{k\geq0}\,\, a_ke^{ik\theta}\qquad\text{and}\qquad
{\bf P}_-\bigl(\Sigma_{k\in\ZZ} a_ke^{ik\theta}\bigr)=
\Sigma_{k<0} a_ke^{ik\theta}
\end{equation}
denote the so-called Riesz projections, that discard the Fourier coefficients
of strictly negative and non-negative index respectively.
Clearly ${\bf P}_+$ (resp. ${\bf P}_-$)  contractively maps $L^2$
onto $H^2$ (resp.   $\bar{H}^{2,0}$) and ${\bf P}_++{\bf P}_-=I$.

Note that, 
by Cauchy's formula, ${\bf P}_\pm(f)$ can be expressed as a Cauchy integral:
\begin{equation}
\label{Cauchyproj}
{\bf P}_+(f)(z)=\frac{1}{2i\pi}\int_{\TT}\frac{f(\zeta)}{\zeta-z}d\zeta, \quad
|z|<1,\qquad
{\bf P}_-(f)(z)=\frac{1}{2i\pi}\int_{\TT}\frac{f(\zeta)}{z-\zeta}d\zeta, \quad
|z|>1.
\end{equation}

The Hardy space $H^\infty=H^\infty(\DD)$ consists of bounded holomorphic 
functions on $\DD$, endowed with the {\it sup} norm.
From \eqref{defintH2} we see that $H^\infty$ embeds contractively in $H^2$, 
in particular each $f\in H^\infty$ has
a non-tangential limit ${f^*}$ on $\TT$. It can be shown that
$\|{f^*}\|_\infty= \|f\|_{H^\infty}$, and that the map
$f\mapsto{f^*}$
is an isometry from $H^\infty$ onto the closed subspace of $L^\infty$
comprised of functions whose Fourier coefficients of strictly negative
index do vanish. 
Again we identify $H^\infty$ with this subspace.
Likewise, the space $\bar{H}^{\infty,0}$ of bounded holomorphic functions 
vanishing at infinity in 
$\CC\setminus\overline{\DD}$ identifies {\it via} non-tangential limits
with the closed subspace of $L^\infty$
consisting of functions whose Fourier coefficients of non-negative
index do vanish. 
However, in contrast with the situation for $L^2$,
the operators ${\bf P}_{\pm}$ are unbounded on $L^\infty$.

A fundamental fact \cite[ch. II, cor. 5.7]{Garnett} is that each
 nonzero $f\in H^2$ factors uniquely
as $f=jw$ where
\begin{equation}
\label{defext}
w(z)=\exp\left\{\frac{1}{2\pi}\int_0^{2\pi}\frac{e ^{i\theta}+z}
{e^{i\theta}-z}\log|f(e ^{i\theta})|\, d\theta\right\}
\end{equation}
belongs to $H^2$ and is called the {\em outer factor} of $f$, normalized  to be positive at zero,
while
$j\in H^\infty$ has modulus 1 a.e. on $\TT$ and is called the
{\em inner factor} of $f$. The latter decomposes further  as
$j=bS$, where
\begin{equation}
\label{defBlaschke}
b(z)=cz^k\prod_{\zeta_l\neq0}\frac{{-\bar \zeta}_l}{|\zeta_l|}\,\frac{z-\zeta_l}{1-{\bar \zeta}_lz}
\end{equation}
is the {\em Blaschke product}, with zero of multiplicity $k\geq0$ at
the origin, associated to a sequence of points
$\zeta_l\in\DD\setminus\{0\}$
and to a constant $c\in\TT$, while
$$S(z)=\exp\left\{-\frac{1}{2\pi}\int_0^{2\pi}\frac{e ^{i\theta}+z}{e
    ^{i\theta}-z}\, d\mu(\theta)\right\}$$
is the {\em singular inner factor} associated with a positive 
singular measure $\mu$
on $\TT$. 
The $\zeta_l$ are of course the zeros of $f$ in $\DD\setminus\{0\}$, counting
multiplicities by repetition. The total number of zeros, finite or infinite, 
is called the degree of the Blaschke product. When the latter is finite, the
Blaschke product is  rational and the two notions of degree that we introduced do coincide. Throughout,
we let $B_n$ denote
the set of Blaschke products of degree at most $n$. When $m<\infty$, note that $B_m\subset\cR_{m,m}$ 
is comprised of rational functions of degree at most $m$ which are
analytic in $\DD$ and have unit modulus everywhere on $\TT$. 
Alternatively, $B_m$ consists of
functions of the form $q_k/\widetilde{q}_k$  where 
$q_k\in \cP_k$ has exact degree $k\leq m$ and all its roots in $\DD$.
Clearly, $B_m$ is included in the unit sphere of both $H^2$ and~$H^\infty$. 
If the degree is infinite,
the convergence of the product in (\ref{defBlaschke})
is equivalent to the condition
\begin{equation}
\label{Bcond}
\sum_l(1-|\zeta_l|)<\infty
\end{equation}
 which holds automatically when $f\in H^2$. 
That $w(z)$ is well-defined rests on the fact that
$\log|f|\in L^1$ if $f\in H^2\setminus\{0\}$.
A function $f\in H^2$ with inner-outer factorization $f=jw$
lies in $H^\infty$ if, and only if $w\in L^\infty(\TT)$.
For simplicity, we often say that a function is outer (resp. inner) 
if it is equal to its outer (resp. inner) factor.

We further set 
\[
H_m^2:=\{\frac{g}{q_m}:\ g\in H^2,\ q_m\in \cP_m,\ q_m(z)\neq0\ \mathrm{for}\  |z|\geq1 \}.
\]
Members of $H_m^2$ identify in $L^2$ with non-tangential limits of 
meromorphic functions  with at most $m$ poles in $\DD$ 
(counting multiplicities) whose
$L^2$-means over $\{|z|=r\}$ remain eventually bounded 
as $r\to 1^-$. 
Two equivalent descriptions of $H_m^2$ are useful:
on the one hand we have that $H^2_m=B_m^{-1}H^2$, the set of quotients of $H^2$-functions by
Blaschke products of degree at most $m$, on the other hand we
get by pole-residue decomposition that
$H_m^2=H^2+(\cR_{m-1,m}\cap \bar H^{2,0})$.
In a similar vein, we put
\[
H_m^\infty:=H^2_m\cap L^\infty=B_m^{-1}H^\infty=\{\frac{g}{q_m}:\ g\in H^\infty,\ q_m\in \cP_m,\ q_m(z)\neq0\ \mathrm{for}\  |z|\geq1\}
\]
for the set of meromorphic functions with at most $m$ poles in $L^\infty$.
\section{Best rational and meromorphic approximation on the circle}
\label{Rat}
We dwell on the discussion in \cite{BCQ}.
For $n\geq1$ an integer, 
the problem of best rational approximation of degree at most $n$ in $L^2$ is:

\noindent{\bf Problem R($\boldsymbol{n}$):} \emph{Given $h\in L^2$,  find $r^*\in \cR_{n,n}$ such that
\[\|h-r^*\|_2=\min_{r\in \cR_{n,n}}\|h-r\|_2.\]}

If we write in view of (\ref{decompP}) the decomposition
$h=h_1+h_2$ with $h_1\in H^2$, $h_2\in\bar{H}^{2,0}$, and invoke
partial fraction expansion to decompose an arbitrary  $r\in \cR_{n,n}$
as $r_1+r_2$ with $r_1\in H^2$, $r_2\in\bar{H}^{2,0}$,
$\mbox{deg\ }r_1+\mbox{deg\ }r_2\leq n$, then by Parseval's relation
\[\|h-r\|_2^2=\|h_1-r_1\|_2^2+\|h_2-r_2\|_2^2\]
so that  R($n$) reduces, modulo a combinatorial allocation of the degrees
of $r_1$ and $r_2$ ($n+1$ choices), to a pair of problems of the
following types:

\noindent{\bf Problem RA($\boldsymbol{n}$):} \emph{Given $f\in H^2$, find $r^*\in \cR_{n,n}\cap H^2$ 
such that
\[\|f-r^*\|_2=\min_{r\in \cR_{n,n}\cap H^2}\|f-r\|_2.\]}

\noindent{\bf Problem RAB($\boldsymbol{n}$):}
\emph{Given $f\in \bar{H}^{2,0}$,  find $r^*\in \cR_{n-1,n}\cap \bar{H}^{2,0}$ 
such that
\[\|f-r^*\|_2=\min_{r\in \cR_{n-1,n}\cap \bar{H}^{2,0}}\|f-r\|_2.\]}

In ``RA($n$)'' and ``RAB($n$)'',
the letter $"A"$ is mnemonic for ``analytic''
and  $"B"$  is mnemonic
for ``bar''.

Problem RA($n$) is in fact equivalent to RAB($n$).
For we can parameterize  $r\in \cR_{n,n}\cap H^2$
as $r(0)+zr_3$ where $r(0)\in \RR$ and $r_3\in \cR_{n-1,n}\cap H^2$
vary independently, hence by Parseval's theorem
\[\|f-r\|_2^2=|f(0)-r(0)|^2+\|(f-f(0))-zr_3\|_2^2\] 
so that $r(0)=f(0)$ is the optimal choice.
Then, since multiplying by $1/z$ preserves the modulus on $\TT$,
we find upon replacing $f$ by $(f-f(0))/z$ 
that Problem RA($n$) is equivalent to the normalized version:

\noindent{\bf Problem RAN($\boldsymbol{n}$):} \emph{Given $f\in H^2$,  find 
$r^*\in \cR_{n-1,n}\cap H^2$ such that
\[\|f-r^*\|_2=\min_{r\in \cR_{n-1,n}\cap  H^2}\|f-r\|_2.\]}

\label{remarkReductionFromRAtoRAB}Now, applying the check operation defined in \eqref{defcheck} which
preserves $\cR_{n-1,n}$ and the degree, RAN($n$) is 
mechanically  equivalent to RAB($n$), which proves the desired equivalence.
Note that when passing from RA($n$) to RAB($n$), the initial 
$f\in H^2$, to be approximated from $\cR_{n,n}\cap H^2$,
gets transformed into  the function
$(f-f(0))^\sharp=\overline{f(1/\bar z)}-\overline{f(0)}\in \bar{H}^{2,0}$, to be approximated from
$\cR_{n-1,n}\cap\bar{H}^{2,0}$.

Finally, we also state the best meromorphic
approximation problem 
with at most $n$ poles in $L^2$: 

\noindent{\bf Problem MA($\boldsymbol{n}$):} 
\emph{Given $f\in L^2$,  find $g^*\in H^2_n$ 
such that
\[\|f-g^*\|_2=\min_{g\in H^2_n}\|f-g\|_2.\]} 

Problem MA($n$) is also equivalent to
RAB($n$). Indeed, since $H_n^2=H^2+(\cR_{n-1,n}\cap\bar{H}^{2,0})$,
it holds by  orthogonality of 
$H^2$ and  $\bar{H}^{2,0}$ that the $H^2$-component of a minimizer
in MA($n$) must be ${\bf P}_+(f)$ while the 
$\bar{H}^{2,0}$-component of this minimizer is a solution to
RAB($n$), with $f$ replaced by ${\bf P}_-(f)$. 

It is known that RAB($n$) has a solution  which needs not be 
unique, and every solution has exact degree $n$ unless $f$ is rational 
of degree at most $n-1$ \cite{Erohin,Levin,Bar86}.

We shall write $d_2(f,\cR_{n-1,n})$ (resp. $d_2(f,\cR_{n,n})$)
for the distance from $f$ to $\cR_{n-1,n}$ (resp. $\cR_{n,n}$) in
$L^2$. Thus,  if  $f\in \bar{H}^{2,0}$, then $d_2(f,\cR_{n-1,n})$ is
both the value of  Problem RAB($n$) and of Problem MA($n$); and if
$f\in H^2$, then  $d_2(f,\cR_{n-1,n})$  (resp.  $d_2(f,\cR_{n,n})$)
is the value of problem RAN($n$) (resp. RA($n$)).
The value of MA($n$)  is denoted by $d_2(f, H^2_n)$.

Problems R($n$) and MA($n$) have natural analogs
in $L^\infty$ provided that $f$ lies in that space:  we denote by  $d_\infty(f,\cR_{n-1,n})$ (resp. $d_\infty(f,\cR_{n,n})$)
the distance from $f$ to $\cR_{n-1,n}$ (resp. $\cR_{n,n}$)  in $L^\infty$;
and by $d_{\infty}(f, H^\infty_n)$   the $L^\infty$-distance from $f$ to 
$H^\infty_n$. In $L^\infty$, however, no decomposition like \eqref{decompP}
is valid and no splitting of the rational approximation problem into analogs of
RA($n$) and RAB($n$) is available.

Of all these approximations problems, the $L^\infty$ analog of MA($n$) may seem most difficult
but still it is the only one known to reduce
to matrix analysis, namely to singular value decomposition
(at least when $f$ is rational). This is a consequence of the
Adamjan-Arov-Krein theory (in short: AAK theory), a full account of which may be found in
\cite{AAK} or \cite[Chapter 4]{Peller}. For  $f\in L^\infty$, these authors 
consider  the Hankel operator $\Gamma_f$ with symbol $f$, which is the bounded operator defined as
\begin{equation}
\label{Hankelinf}
\begin{array}{lll}
\Gamma_{f}:  H^{2} & \longrightarrow &  \bar{H}^{2,0} \\
\ \ \ \ \ \ \      v &          \mapsto & {\bf P}_-(fv), \\
\end{array}
\end{equation}
together with its singular values $s_0(\Gamma_f)\geq s_1(\Gamma_f)\geq s_2(\Gamma_f)\cdots$ given by
\begin{equation}
\label{singdefg}
s_n(\Gamma_f):=\inf \bigl\{ 
|||\Gamma_f-L|||,~~L\in\mathcal{L}_n\bigr\},
\end{equation}
where $\mathcal{L}_n$ is the collection of linear operators
$H^2\to \bar{H}^{2,0}$ with rank not exceeding $n$ 
and $|||\cdot|||$ denotes the operator norm. Now, the AAK-theorem asserts that
\begin{equation}
\label{egAAK}
d_{\infty}(f,H^\infty_n)=s_n(\Gamma_f).
\end{equation}
Moreover, if $f\in C(\TT)$  then $\Gamma_f$ is
compact,
and
the unique best approximant $g_n$ to $f$ from $H_n^\infty$ in $L^\infty$ is
\begin{equation}
  \label{bmai}
  g_n=\frac{{\bf P}_{+}\bigl(fv_n\bigr)}{v_n}
\end{equation}
where $v_n$ is any singular vector associated with $s_n(\Gamma_f)$
(any eigenvector of $\Gamma_f^*\Gamma_f$ with eigenvalue 
$(s_n(\Gamma_f))^2$).

When $\Gamma_f$ is compact, recall also  that the Courant $\max\min$ principle \cite[sec. 22.11a]{ZeidlerIIA} applies:
\begin{equation}
\label{CPHilb}
s_n(\Gamma_f)=\max_{V\in\mathcal{V}_{n+1}} \min_{\stackrel{v \in V}{\|v\|_2=1}}\|\Gamma_f(v)\|_2,
\end{equation}
where $\mathcal{V}_{n+1}$ is the collection of linear subspaces of $H^2$
of complex dimension at least $n+1$.

Let us  turn to problem RAB($n$), which is equivalent to MA($n$) as pointed out
already. That is, let us move from meromorphic approximation in $L^\infty$ to meromorphic approximation in $L^2$.
In this context, for $f\in  \bar{H}^{2,0}$, we define again a  Hankel operator $A_f$ with symbol $f$ 
by 
\begin{equation}
\label{Hankel21}
\begin{array}{lll}
A_{f}:  H^{\infty} & \longrightarrow &  \bar{H}^{2,0} \\
\ \ \ \ \ \ \      v &          \mapsto & {\bf P}_-(fv). \\
\end{array}
\end{equation}
Although the definitions of $A_f$ and $\Gamma_f$ are formally the same,
observe that the domains in~\eqref{Hankel21} and~\eqref{Hankelinf}
are different.
Evidently,  $A_f$ is continuous and  $|||A_f|||=\|f\|_2$. In fact $A_f$ is compact, because it is a limit of finite rank operators: indeed, since rational functions are dense in $\bar{H}^{2,0}$, we may find $r_n\in\mathcal{R}_{n-1,n}$ arbitrary close to $f$, and use Kronecker's theorem that Hankel operators with rational symbol have finite rank \cite{Peller}.

The Theorem below  is a special case of {\rm \cite[thm. 8.1]{BS02}},
but
granted its importance in the present paper and its short proof we
nevertheless provide an argument for the ease of the reader.
\begin{theorem}
\label{lemH1}
For $f\in \bar{H}^{2,0}$, it holds that
\begin{equation}
\label{minHankB1}
d_2(f,\cR_{n-1,n})=\min_{\substack{b_n\in B_n\\ \deg(b_n) = n}}\|A_f(b_n)\|_2.
\end{equation}
Moreover, a rational function $r_n\in\cR_{n-1,n}$ is a solution to 
{\rm RAB(}$n${\rm )} if, and only if \begin{equation}
  \label{bal2}
  r_n=\frac{{\bf P}_+(fb_n)}{b_n}
\end{equation}
where $b_n$ is any minimizing Blaschke product in
the right hand side of \eqref{minHankB1}.

\end{theorem}
\begin{proof}
Let us parameterize 
$r\in \cR_{n-1,n}\cap \bar{H}^{2,0}$ as $r=p_{n-1}/q_n$
where $p_{n-1}$ ranges over $\cP_{n-1}$ and $q_n$ ranges over those polynomials
in $\cP_n$ whose degree is exactly $n$ and whose roots all lie in $\DD$.
Then $q_n/\widetilde{q}_n\in B_n$, so multiplication by  $q_n/\widetilde{q}_n$  is an isometry in $L^2$, and since 
$p_{n-1}/\widetilde{q}_n\in H^2$ we get by
orthogonality of  $H^2$ and  $\bar{H}^{2,0}$ on using that ${\bf P}_++{\bf P_-}=I$: 
\begin{equation}
\label{multB}
\|f-\frac{p_{n-1}}{q_n}\|_2^2=\|f\frac{q_n}{\widetilde{q}_n}-
\frac{p_{n-1}}{\widetilde{q}_n}\|_2^2=
\|{\bf P}_-(f\frac{q_n}{\widetilde{q}_n})\|_2^2+
\|{\bf P}_+(f\frac{q_n}{\widetilde{q}_n})
-\frac{p_{n-1}}{\widetilde{q}_n}\|_2^2.
\end{equation}
Clearly the product of a $\bar{H}^{2,0}$-function by a polynomial in $\cP_n$
yields a member of $z^n\bar{H}^{2,0}$. Therefore
\begin{equation}
\label{minp}
\widetilde{q}_n{\bf P}_+(f\frac{q_n}{\widetilde{q}_n})=
fq_n-\widetilde{q}_n{\bf P}_-(f\frac{q_n}{\widetilde{q}_n})\in 
z^n\bar{H}^{2,0}\cap H^2=\cP_{n-1},
\end{equation}
implying that $p_{n-1}=\widetilde{q}_n{\bf P}_+(fq_n/\widetilde{q}_n)$
is the minimizing choice in \eqref{multB} for fixed $q_n$.
Consequently 
\begin{equation}
\label{vpreHank}
\min_{r\in \cR_{n-1,n}\cap \bar{H}^{2,0}}\|f-r\|_2=
\min_{\substack{q_n\in \cP_n, \mathcal{Z}(q_n)\subset\DD \\ \deg(q_n) = n}}
\|{\bf P}_-(f\frac{q_n}{\widetilde{q}_n})\|_2=
\min_{\substack{b_n\in B_n\\ \deg(b_n) = n}}
\|{\bf P}_-(fb_n)\|_2=\min_{\substack{b_n\in B_n\\ \deg(b_n) = n}}\|A_f(b_n)\|_2.
\end{equation}
That the infimum is indeed attained in \eqref{vpreHank}
follows from \eqref{multB} and the fact that RAB($n$) has a solution.
\end{proof}
\begin{remark}
  One may rather parameterize $r\in \cR_{n-1,n}\cap \bar{H}^{2,0}$ as $r=p_{d-1}/q_d$
  where $d \le n$, the numerator $p_{d-1}$ ranges over $\cP_{d-1}$ and the denominator $q_d$ ranges over those polynomials in $\cP_d$ whose degree is exactly $d$ and whose roots all lie in $\DD$. Replaying the proof with this parameterization yields
  \begin{equation}
\label{minHankB2}
d_2(f,\cR_{n-1,n})=\min_{b_n\in B_n}\|A_f(b_n)\|_2,
\end{equation}
with \eqref{bal2} still holding for any minimizing Blaschke product $b_n$ in the right hand size of \eqref{minHankB2}. And indeed, if $f$ does not belong to $\cR_{n-1,n}$ we already recalled that any solution $r_n \in \cR_{n-1,n}$ to RAB($n$) is known to have exactly $n$ poles (counted with multiplicity): in view of~\eqref{bal2} this show that $b_n$ sought on the whole $B_n$ has to be of degree $n$ in order to be a minimizer.
\end{remark}

We digress at this point to stress how Theorem \ref{lemH1} parallels
AAK theory: we may define the singular values  $s_n(A_f)$ by
\eqref{singdefg}, only with $A_f$ instead of $\Gamma_f$ and $H^\infty$
instead\footnote{In \cite[Section 8]{BS02} it is required in this case that members of $\mathcal{L}_n$ should be weak-$*$ continuous $ H^{\infty} \to\bar{H}^{2,0}$.} of $H^2$,
and still  $d_2(f,\cR_{n-1,n})=s_n(A_f)$; {\it i.e.} the natural analog of \eqref{egAAK} holds. However,
a direct analog of the Courant max-min principle \eqref{CPHilb}
does not hold when we replace $\Gamma_f$ by $A_f$ and $H^2$ by $H^\infty$ in the definition of $\mathcal{V}_{n+1}$.
But still, a nonlinear analog of the max-min principle does hold if the collection of all Euclidean unit spheres of real dimension $2n+1$ in
$H^2$, over which the maximizing step is made in \eqref{CPHilb}, gets replaced by the collection of all weak-$*$ closed symmetric subsets of the unit sphere
of $H^\infty$ whose genus\footnote{The genus of a symmetric set $K$ (meaning that $-v\in K$ when $v\in K$) is the smallest positive integer for which there is  an odd continuous mapping $K\to\RR^n\setminus\{0\}$ \cite{ZeidlerIIA}; by the Borsuk-Ulam theorem, Euclidean spheres of real dimension $2n+1$ have genus $2n+2$.} is at least $2n+2$. Specifically,
if we let $\mathcal{K}_{2n+2}^\infty$ denote the collection of all such
sets, it holds that
\begin{equation}
\label{minmax1}
s_n(A_f)=\max_{K \in {\cal K}_{2n+2}^\infty} \min_{u \in K}\|A_f(u)\|_2,\qquad n=0,1,2,...
\end{equation}
and \eqref{minHankB2}  indicates that $B_n$ is an optimal weak-$*$ closed subset
of the unit sphere of $H^\infty$ with genus no less than (in fact equal to) $2n+2$ in \eqref{minmax1}. Moreover, if $b_n$ is a minimizing
Blaschke product  in \eqref{minmax1}, formula \eqref{bal2} stands analog to \eqref{bmai}.
We refer the reader to \cite{BS02} for an account of best meromorphic approximation in $L^p$ for $2\leq p\leq\infty$ along these lines.

\section{$H^2$  rational approximation of Blaschke products: lower bounds}
\label{Lbounds}

When $h$ is continuous of constant modulus on $\TT$, the unique best
approximant to $h$ from $\mathcal{R}_{m,n}$ in $L^\infty$
must be zero unless $-n\leq W(h)\leq m$. Indeed, if $r_{m,n}\in \mathcal{R}_{m,n}$
is such that $|h-r_{m,n}|<|h|$ on $\TT$, then $[0,1]\ni t\mapsto h+t(r_{m,n}-h)$
is a nevervanishing homotopy between $h$ and $r_{m,n}$, therefore these must have the same index and yet  $-n\leq W(r_{m,n})\leq m$, by \eqref{indmero}.
In particular, for $f\in B_N$ with $N$ finite, the solution to problem RA($n$)  
is zero when $N>n$; that is to say, $d_\infty(f,\cR_{n,n})=1$  if $n<N<\infty$.
The same is in fact true when $N=\infty$, though in this case
$f$ is no longer continuous and the argument needs adjustment, see \cite[Lemma 1]{BCQ}.

The situation is different if we consider best rational approximation in $L^2$. Indeed, if $f\in B_N$ then  $d_2(f,\cR_{n,n})<1$, because
any solution to RAB($n$) has exact degree $n$ unless $N<n$, as mentioned 
in Section \ref{Rat}; hence, zero is never a  best approximant. Of course, since approximation in $L^\infty$ is ineffective, one may surmise that
$d_2(f,\cR_{n,n})$ cannot be small when $f\in B_N$ and $n$ is much smaller than $N$,
but this suspicion calls for quantification. Now, if $(\zeta_j)_{1\leq j\leq N}$
is the sequence of zeros of $f\in B_N$ with moduli arranged in non-decreasing order
($N$ may be infinite), it can be shown building
on Theorem \ref{lemH1} and AAK theory that $d_2(f,\cR_{n,n})$ is no less
than $\sqrt{(1-|\zeta_{n+1}|^2)/(n+1)}$ provided that $n<N$; see
\cite[Theorem 5]{BCQ}. Since $(1-|\zeta_{n+1}|)$ can be the general term of
a convergent series with positive, non-increasing but otherwise arbitrary
summands, this estimate
gives one an appraisal of how slow the decay of $d_2(f,\cR_{n,n})$ to zero
can be as $n\to\infty$ when $N$ is infinite. Still it is unclear how sharp this
lower bound, and when $N$ is finite it gives no information
on the speed of $L^2$-approximation as $n$ increases to $N$ (when $n>N$, the best approximant is of course $f$ itself). This is the question that we take up now. We first consider the particular case where $f=z^N$, which is of special
significance from a system-theoretic viewpoint, because it quantifies how
poorly discrete-time delay systems  lend themselves to model reduction.

\begin{theorem}
  \label{initp}
  Let $0\leq n<N$ be integers. Then,
  \begin{equation}
    \label{lbinfi}
    d_2(z^N,\cR_{n,n})\geq \sqrt{1-\frac{n}{N}}.
    \end{equation}
  \end{theorem}
  \begin{proof}
    Define $f \in H^2$ as $f : z \mapsto z^N$ and observe, in connection with the discussion after the definition of Problem RAN($n$) in Section \ref{Rat}, that $(f-f(0))^\sharp = f^\sharp = z^{-N}$, whence
    \begin{equation}
      \label{equiv_rab_ran}
      d_2(z^N,\cR_{n,n}) = d_2(f, \cR_{n,n}) = d_2( (f-f(0))^\sharp, \cR_{n-1,n}) = d_2(z^{-N},\cR_{n-1,n}).
    \end{equation}
    It follows from Theorem \ref{lemH1} that
      \begin{equation}
    \label{expverr1}
d_2(z^{-N},\cR_{n-1,n})=\min_{b_n\in B_n}\|A_{z^{-N}}(b_n)\|_2=\min_{b_n\in B_n}
    \|{\bf P}_-(z^{-N}b_n)\|_2.
  \end{equation}
  Pick $b_n\in B_n$ and write its Fourier expansion as $b_n=\sum_{k=0}^{+\infty}a_k z^k$.  We further denote $d \le n$ the degree of $b_n$. Then, \eqref{defRiesz} and Parseval's identity yield 
  \begin{equation}
    \label{estgen}
\|{\bf P}_-(z^{-N}b_n)\|_2^2=\sum_{k=0}^{N-1} |a_k|^2.
\end{equation}
In another connection, since $e^{-iN\theta}b_n(e^{i\theta})=\sum_{k=0}^{+\infty}a_k e^{i(k-N)\theta}$,
we get from \eqref{indmero} and \eqref{degreeW12} that
\begin{equation}
  \label{detind1}
  W(z^{-N}b_n)=d-N=\sum_{k=0}^{N-1}(k-N)|a_k|^2+\sum_{k=N}^{+\infty}(k-N)|a_k|^2;
\end{equation}
note that \eqref{detind1} is valid even if $b_n$ has a zero at zero, because the quantity $d-N$ remains unchanged if $z^N$ and $b_n$ undergo  the same drop
in degree.
From \eqref{detind1} we deduce at once that
\[
N-n \leq N-d \leq \sum_{k=0}^{N-1}(N-k)|a_k|^2,
  \]
  whence
  \begin{equation}
    \label{conppur}
        1-\frac{n}{N}\leq \sum_{k=0}^{N-1}(1-\frac{k}{N})|a_k|^2\leq
    \sum_{k=0}^{N-1}|a_k|^2.
  \end{equation}
  Inequality \eqref{lbinfi} now follows from \eqref{conppur}, \eqref{estgen}, \eqref{expverr1} and \eqref{equiv_rab_ran}.
    \end{proof}
    \begin{remark}
      Equality does not hold in general in Equation~\eqref{lbinfi}: one can see this in the case $n=1$ already, which can be tackled explicitly and is worth mentioning. Using the formulation as Problem RAB($n$), one is reduced to find $\eta \in \CC$ and $\omega \in \DD$ such that $\delta(\omega, \eta) = \left\|\frac{\eta}{z-\omega}-\frac{1}{z^N}\right\|_2^2$ be minimal. Since $\frac{\eta}{z-\omega} = \frac{\eta}{z}\sum_{k=0}^\infty \frac{\omega^k}{z^k}$, by Parseval's equality, one gets
      \begin{align*}
        \delta(\omega, \eta) &= |\eta \omega^{N-1} - 1|^2 + \sum_{\substack{k=0\\k\neq N-1}}^\infty |\eta \omega^k|^2\\
                              &= (|\eta \omega^{N-1}|^2 - 2\Re(\eta \omega^{N-1}) + 1 ) + \sum_{\substack{k=0\\k\neq N-1}}^\infty |\eta \omega^k|^2\\
                              &= 1 - 2\Re(\eta \omega^{N-1}) + \frac{|\eta|^2}{1-|\omega|^2}.
      \end{align*}
For a given $\omega = r e^{i\theta} \in \DD\backslash\{0\}$, the value of $\eta$ that minimizes $\delta(\omega, \eta)$ is $\eta = r' e^{i (1-N)\theta}$ where $r'$ minimizes the function $f : x \mapsto 1 - 2r^{N-1}x + \frac{1}{1-r^2}x^2$ over $\RR_+$. The minimum of $f$ on the whole real line is reached at $x = r^{N-1}(1-r^2) > 0$, so this is the optimal value of $r'$. The corresponding value of $\delta(\omega, \eta)$ is $1-r^{2(N-1)}(1-r^2)$. Therefore, the problems boils down to minimize the function $g : x \mapsto 1-x^{2N-2}(1-x^2)$ on $(0,1)$ which is obtained for $x = \sqrt{(N-1)/N}$. We finally get the minimal value of $\delta(\omega, \eta) = 1-\frac{1}{N}\,\left(\frac{N-1}{N}\right)^{N-1}$ and therefore \[ d_2(z^{-N},\cR_{0,1}) = \left(1-\frac{1}{N-1}\,\left(1-\frac{1}{N}\right)^{N}\right)^{1/2}\] which is strictly greater than $\sqrt{1-\frac{1}{N}}$. Notice however that, asymptotically when $N$ grows large, one has $d_2(z^{-N},\cR_{0,1}) \sim \left(1-\frac{1}{2eN}\right)$ while the lower bound is $\sqrt{1-1/N} \sim \left(1-\frac{1}{2N}\right)$ which captures the rate within a moderate constant factor.
    \end{remark}
    We turn to the case where $f$ is an element of $B_N$ different from $z^N$,
    which is more involved. We shall introduce 
    the maximum modulus of the zeros of $f$ as a parameter,  and resort to estimates proven in  the appendix to establish the  lemma below. To this end, we
    recap here some notation used there: we consider $\lambda \in (0,1)$ and define $\alpha_0 = \frac{1-\lambda}{1+\lambda}$. In the sequel, we consider an arbitrary value $\alpha \in (0, \alpha_0)$ and we put
    \begin{equation}
        \label{defs}
      s^*(\lambda; \alpha) = \frac{\alpha^{-1}-1+(\alpha^{-1}+1)\lambda^{2}}{2\lambda\alpha^{-1}}
        +\sqrt{\left(\frac{\alpha^{-1}-1+(\alpha^{-1}+1)\lambda^{2}}{2\lambda\alpha^{-1}}\right)^{2}-1}.
      \end{equation}
      It is proven in the appendix that $s^*(\lambda; \alpha) \in (1, 1/\lambda)$. % hence $0<s^*-1<\lambda^{-1}-1=(1-\lambda)/\lambda$, in particular $s^*\downarrow 1$ as $\lambda\uparrow 1$ while $\alpha_0=(1-\lambda)/(1+\lambda)\downarrow0$.
We also recall the polylogarithm $\operatorname{Li}_{\beta}$ of index $\beta$ which is the function defined for $|x|<1$ by 
    $\operatorname{Li}_{\beta}(x):=\sum_{\ell\ge1}x^{\ell}/\ell^{\beta}$.
    \begin{lem}
      \label{coeffprod}
      Let $0\leq n<N$ be integers, $f\in B_N$ of degree $N$ and $b_n\in B_n$. Assume that the modulus of the zeros of $f$ does not exceed $\lambda\in(0,1)$, and put $\alpha_0:=\frac{1-\lambda}{1+\lambda}$. If $\sum_{k\in\ZZ}a_ke^{ik\theta}$ is the Fourier expansion of $f^\sharp b_n$, then
      for each $\alpha\in(0,\alpha_0)$ the following inequality holds:
      \begin{equation}
        \label{estcoeff}
        \forall \beta\in[0,1],\, \forall k \in \ZZ \text{ such that } k\leq-\frac{N}{\alpha}, \quad \left|a_k\right|\leq \frac{1}{(s^*)^{|k|-1}}\left(\frac{s^*-\lambda}{1-\lambda s^*}\right)^N K(n,\lambda ; \alpha, \beta),
      \end{equation}
      where $s^* := s^*(\lambda; \alpha)$ is the quantity defined in~\eqref{defs} and $K(n,\lambda ; \alpha, \beta)$ is defined as
      \begin{equation}
        \label{defK}
        K(n,\lambda ; \alpha, \beta) = (n+1)^{\beta/2}\sqrt{\operatorname{Li}_{\beta}\!\big(s^*(\lambda; \alpha)^{-2}\big)}
      \end{equation}
    \end{lem}
    \begin{proof}
        Denoting by $\widehat{f}(j)$  and $\widehat{b}_n(j)$, $j\in\mathbb{N}$,
        the Fourier coefficients of $f$ and $b_n$ respectively (the Fourier coefficients of negative index do vanish), one can check that the Fourier coefficient of $f^\sharp b_n$ with index $k<0$  may be written as
        \[
a_k=\sum_{j=-\infty}^k\overline{\widehat{f}(-j)}\,\widehat{b}_n(k-j)
\]
and re-indexing the convolution with $m:=k-j\geq 0$ yields 
\[
a_k=\sum_{m\ge 0}\overline{\widehat{f}(-k+m)}\,\widehat{b}_n(m).
\]
Let us pick  $\beta\in[0,1]$ and apply the Cauchy--Schwarz inequality with weight $\{(m+1)^{\beta}\}_{m\in\mathbb{N}}$, so as  to obtain
\begin{equation}
  \label{wCScoeffprod-beta}
|a_k|\le \left(\sum_{m\ge 0}(m+1)^{\beta}\abs{\widehat{b}_n(m)}^2\right)^{1/2}
\left(\sum_{m\ge 0}\frac{\abs{\widehat{f}(-k+m)}^2}{(m+1)^{\beta}}\right)^{1/2}.
\end{equation}
Since
$\sum_{m\ge0}\abs{\widehat b_n(m)}^2=1$  and
$\sum_{m\ge0}(m+1)\abs{\widehat b_n(m)}^2\leq n+1$ with equality if $b_n$ has exact degree $n$ (by \eqref{degreeW12}), we get from H\"older's inequality that 
\[
\sum_{m\ge 0}(m+1)^{\beta}\abs{\widehat{b}_n(m)}^2
\le \Big(\sum_{m\ge0}\abs{\widehat b_n(m)}^2\Big)^{1-\beta}
\Big(\sum_{m\ge0}(m+1)\abs{\widehat b_n(m)}^2\Big)^{\beta}
\leq(n+1)^{\beta}.
\]
Consequently, \eqref{wCScoeffprod-beta} implies for any $\beta\in[0,1]$ that
\begin{equation}
  \label{wCScoeffprod-beta2}
|a_k|\le (n+1)^{\beta/2}\left(\sum_{m \ge 0}\frac{\abs{\widehat{f}(-k+m)}^2}{(m+1)^{\beta}}\right)^{1/2}.
\end{equation}
Now, Equation~\eqref{majFTbn} in Proposition~\ref{estFTnpa} in the appendix, provides us with a geometric bound for $\abs{\widehat f(j)}$ as soon as $j \ge N/\alpha$, namely $\abs{\widehat f(j)}\leq C (s^{*})^{-j}$ where $C$ is a short notation for
\[
C=\left(\frac{s^*-\lambda}{1-\lambda s^*}\right)^N.
\]
Since $k\leq -N/\alpha$ by assumption, the bound applies for all terms appearing in the sum between parentheses on the right hand side of \eqref{wCScoeffprod-beta2} and
therefore
\begin{equation}
  \label{estcoeff-weighted-beta}
  |a_k| \le (n+1)^{\beta/2} C \left( \sum_{m\ge 0} \frac{(s^*)^{2(k-m)}}{(m+1)^\beta}\right)^{1/2}
  = (n+1)^{\beta/2} \left(\frac{s^*-\lambda}{1-\lambda s^*}\right)^N (s^*)^{k+1}\,\sqrt{\operatorname{Li}_{\beta}\!\big((s^*)^{-2}\big)}, 
\end{equation}
for all $k \le -N/\alpha$ and all $\beta \in [0,1]$. This achieves the proof since $k = -|k|$.
\end{proof}
\begin{remark}
  \label{diffSchwarz}
  The Cauchy--Schwarz inequality in \eqref{wCScoeffprod-beta}
  is far from being an equality in most cases, which is  detrimental
  to the upper   bound in    \eqref{estcoeff}.

  If one is  willing to choose as parameters the zeros of $f$ and not just their maximum modulus $\lambda$, an estimate based on the partial fraction
  expansion of $f^\sharp b_n$   may be sharper.
\end{remark}
Lemma \ref{coeffprod} calls for determining the value $\beta$ that minimizes $K(n, \lambda; \alpha, \beta)$ in~\eqref{estcoeff}. To guess, observe since  $\operatorname{Li}_0(x)=x/(1-x)$
and $\operatorname{Li}_1(x)=-\log(1-x)$ that Equation~\eqref{defK} defining $K$ becomes
\begin{align}
 \frac{1}{\sqrt{(s^*)^2-1}} &\text{ for }\beta=0,  \label{estcoeff1}\\[0.2cm]
  (n+1)^{1/2}  \sqrt{\log\frac{(s^*)^2}{(s^*)^2-1}}  &\text{ for }\beta=1 \label{estcoeff2}.
  \end{align}
  Glancing at \eqref{estcoeff1}-\eqref{estcoeff2},  it appears that the best choice for $\beta$ will much depend on $\lambda$ and $\alpha$ (because $s^*$ does) as well as  $n$. To investigate the matter further, consider for simplicity the quantity $\log\left( K(n,\lambda; \alpha, \beta)^2\right)$.

  One has \[\log\left( K(n,\lambda; \alpha, \beta)^2\right) = \beta\log(n+1)+ \log\big(\Phi(\beta)\big),\] where $\Phi(\beta) = \operatorname{Li}_{\beta}\big((s^*)^{-2}\big)$. A convenient way of studying $\Phi$ is to
  consider the probability measure on $\NN^\star$ defined by \[\mathbb{P}_{\beta}(\{m\}) = \frac{1}{\Phi(\beta)}\cdot \frac{1}{m^{\beta} (s^*)^{2m}}.  \] On the one hand $\frac{\partial}{\partial \beta}\big(\log\big(\Phi(\beta)\big)\big) = \frac{\Phi'(\beta)}{\Phi(\beta)}$, and on the other hand $\frac{\partial}{\partial \beta}\left(\frac{1}{m^\beta (s^*)^{2m}}\right) = \frac{-\log(m)}{m^\beta (s^*)^{2m}}$. Therefore, one has \[\frac{\partial}{\partial \beta}\big(\log\big(\Phi(\beta)\big)\big) = \frac{1}{\Phi(\beta)}\sum_{m \ge 1} \frac{-\log(m)}{m^\beta (s^*)^{2m}} = -\mathbb{E}_{\beta}\!\big[\log(m)\big] \]
  and
  \[ \frac{\partial}{\partial \beta}\mathbb{P}_{\beta}(\{m\}) = \left(\mathbb{E}_{\beta}\!\big[\log(m)\big] - \log(m)\right)\mathbb{P}_{\beta}(\{m\}). \]
  Finally, one gets
  \[\frac{\partial^2}{\partial \beta^2}\big(\log\big(\Phi(\beta)\big)\big) = -\frac{\partial}{\partial \beta}\mathbb{E}_{\beta}\!\big[\log(m)\big] = -\sum_{m \ge 1} \log(m)\frac{\partial}{\partial \beta}\mathbb{P}_{\beta}(\{m\}) = \mathrm{Var}_{\beta}\big(\log(m)\big)\,>\,0.\]
This proves that $\log(\Phi)$, and consequently also $\log\big(K(n,\lambda; \alpha, \beta)^2\big)$, is strictly convex with respect to $\beta$ on~$[0,1]$. Therefore, the latter admits a unique minimizer $\beta_{\ast}(n,\lambda ; \alpha)$ which is the only value of $\beta$ such that  \[\log(n+1) = \mathbb{E}_{\beta}\!\big[\log(m),\big]\] whenever the solution of this equation lies in $(0,1)$. Since $\mathbb{E}_{\beta}\!\big[\log(m)\big]$ is strictly decreasing with respect to~$\beta$, $\beta_{\ast}(n,\lambda ; \alpha)$ decreases with $n$.
In particular $\beta_{\ast}(n,\lambda ; \alpha)$ decreases to $0$ and eventually becomes $0$
when $n\to\infty$, while $\beta_{\ast}(n,\lambda ; \alpha)=1$  when $n\le n_0$ for some $n_0=n_0(\lambda; \alpha)$.  
Figure~\ref{fig:beta-star} illustrates $\beta_{\ast}(n,\lambda ; \alpha)$ for
some pairs $(\lambda; \alpha)$.

\begin{figure}[ht]
\centering
\includegraphics{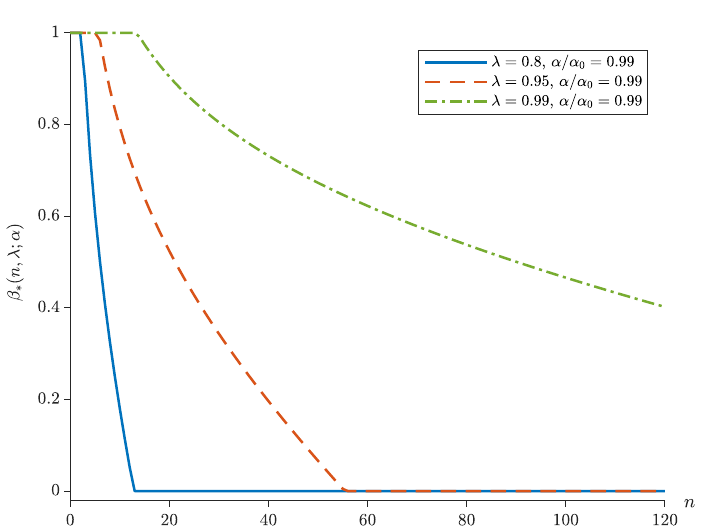}
\caption{\textcolor{black}{Numerically optimal exponent $\beta_{\ast}(n,\lambda ; \alpha)$ minimizing $K(n, \lambda ; \alpha, \beta)=(n+1)^{\beta/2}\sqrt{\operatorname{Li}_{\beta}\big(s^*(\lambda; \alpha)^{-2}\big)}$
for representative values of $(\lambda,\alpha/\alpha_0)$.
The optimizer decreases with $n$ and eventually reaches $0$, in agreement with the previous discussion.}}
\label{fig:beta-star}
\end{figure}

\begin{theorem}
  \label{genBp}
  Let $0\leq n<N$ be integers and $f\in B_N$ of degree $N$. Assume, as in Lemma~\ref{coeffprod}, that the modulus of the zeros of $f$ does not exceed $\lambda\in(0,1)$. We stick to the previous notation,  \emph{i.e.}, $\alpha_0:=\frac{1-\lambda}{1+\lambda}$ and for $\alpha \in (0,\alpha_0)$ the number $s^* = s^*(\lambda; \alpha)$ is as 
   in~\eqref{defs}, while $K(n,\lambda ; \alpha, \beta)$ was given in~\eqref{defK}. For $\alpha \in (0,\alpha_0)$ and  $\beta \in [0,1]$, we further define
  \begin{equation}
    \label{def_Psi}
    \Psi(N, n, \lambda ; \alpha, \beta) = \alpha\left(1-\frac{n}{N}-\frac{1}{N}\left(\frac{s^*-\lambda}{1-\lambda s^*}\right)^{2N}
      K(n,\lambda ; \alpha, \beta)^2\,
    \frac{(s^*)^{2(2-\left\lfloor N/\alpha \right\rfloor)}}{((s^*)^2-1)^2}\right),
  \end{equation}
where $\left\lfloor N/\alpha \right\rfloor$ denotes the integer part of $N/\alpha$. Then, one has the inequality:
  \begin{equation}
    \label{lbgBp}
      \forall \alpha \in (0, \alpha_0),\, \forall \beta \in [0,1],\quad   d_2(f,\cR_{n,n})\geq  \sqrt{
[\Psi(N, n, \lambda; \alpha, \beta)]_+}
  \end{equation}
  with $[\cdot]_+$ to mean $\max\{0, \cdot\}$.
    \end{theorem}
  \begin{proof}
    Define  $h(z):=(f-f(0))^\sharp(z)=\overline{f(1/\bar z)}-\overline{f(0)}\in \bar{H}^{2,0}$, and observe  from the discussion in Section \ref{Rat} after defining problem RAN($n$)
    that
    \begin{equation}
      \label{expverr}
      d_2(f,\cR_{n,n})= d_2(h,\cR_{n-1,n})=\min_{\substack{b_n\in B_n\\\deg(b_n) = n}}\|A_{h}(b_n)\|_2=
\min_{\substack{b_n\in B_n\\\deg(b_n) = n}}\|A_{f^\sharp}(b_n)\|_2=
      \min_{\substack{b_n\in B_n\\\deg(b_n) = n}}
    \|{\bf P}_-(f^\sharp b_n)\|_2,
  \end{equation}
where we used Theorem \ref{lemH1} and the fact that $A_{f^\sharp}=A_{h}$, since $h$ and $f^\sharp$ differ by an additive constant. 
Write the Fourier expansion $f^\sharp b_n(e^{i\theta})=\sum_{k=-\infty}^{+\infty}a_k e^{ik\theta}$, so that
\begin{equation}
    \label{estgenP}
\|{\bf P}_-(f^\sharp b_n)\|_2^2=\sum_{k=-\infty}^{-1} |a_k|^2.
\end{equation}
Because  $f^\sharp=1/f$ has $N$ poles and no zeros in $\DD$, counting multiplicities, we get from \eqref{indmero} and \eqref{degreeW12} that
\[
  W(f^\sharp b_n)=n-N=\sum_{k=-\infty}^{+\infty}k|a_k|^2,
\]
which can be rearranged for any $\alpha\in(0,\alpha_0)$ as
\begin{equation}
  \label{detind}
  N-n=
  \sum_{k=-\infty}^{-\left\lfloor N/\alpha \right\rfloor-1}(-k-\left\lfloor N/\alpha \right\rfloor)|a_k|^2+\left\lfloor N/\alpha \right\rfloor\sum_{k=-\infty}^{-\left\lfloor N/\alpha \right\rfloor-1}|a_k|^2
+\sum_{k=-\left\lfloor N/\alpha \right\rfloor}^{-1}(-k)|a_k|^2
  -\sum_{k=0}^{+\infty}k|a_k|^2.
\end{equation}
Thus, we arrive at the inequality:
\begin{equation}
  \label{detindin}
  N-n\leq
  \sum_{k=-\infty}^{-\left\lfloor N/\alpha \right\rfloor-1}(-k-\left\lfloor N/\alpha \right\rfloor)|a_k|^2+\left\lfloor N/\alpha \right\rfloor\sum_{k=-\infty}^{-1}|a_k|^2.
 \end{equation}

 Next, recall from~\eqref{estcoeff} that
 \begin{equation}
  \label{estcoeff-weighted-beta-m}
\left|a_k\right|\leq \frac{1}{(s^*)^{|k|-1}}\left(\frac{s^*-\lambda}{1-\lambda s^*}\right)^N K(n,\lambda ; \alpha, \beta)
\qquad \mathrm{for\ }k\leq-\frac{N}{\alpha},
\end{equation}
where $K(n, \lambda; \alpha, \beta)$ is as in \eqref{defK}.
By \eqref{estcoeff-weighted-beta-m}, we get that
 \[
   \sum_{k=-\infty}^{-\left\lfloor N/\alpha \right\rfloor-1}(-k-\left\lfloor N/\alpha \right\rfloor)|a_k|^2\leq
  \left(\frac{s^*-\lambda}{1-\lambda s^*}\right)^{2N}K(n,\lambda ; \alpha, \beta)^2 \,(s^*)^2
   \sum_{k=-\infty}^{-\left\lfloor N/\alpha \right\rfloor-1}(|k|-\left\lfloor N/\alpha \right\rfloor)\frac{1}{(s^*)^{2|k|}}.
 \]
 Now, observe that for any $M \in \NN$ and $|z|<1$, \[\sum_{k=M+1}^{+\infty} (k-M)z^k = \frac{z^{M+1}}{(1-z)^2},\]
 to the effect that
 \begin{equation}
   \label{estgeomc}  \sum_{k=-\infty}^{-\left\lfloor N/\alpha \right\rfloor-1}(-k-\left\lfloor N/\alpha \right\rfloor)|a_k|^2\leq \left(\frac{s^*-\lambda}{1-\lambda s^*}\right)^{2N} K(n,\lambda ; \alpha, \beta)^2 
   \left(\frac{(s^*)^{4-2\left\lfloor N/\alpha \right\rfloor)}}{((s^*)^2-1)^2}\right).
 \end{equation}
 Substituting the estimate \eqref{estgeomc} in \eqref{detindin} while dividing by $\left\lfloor N/\alpha \right\rfloor$ yields
\begin{equation}
  \label{inegcd}
  \frac{1}{\left\lfloor N/\alpha \right\rfloor} \left(N-n-\left(\frac{s^*-\lambda}{1-\lambda s^*}\right)^{2N}
     K(n,\lambda ; \alpha, \beta)^2 
    \frac{(s^*)^{2(2-\left\lfloor N/\alpha \right\rfloor)}}{((s^*)^2-1)^2}\right)  \leq
  \sum_{k=-\infty}^{-1}|a_k|^2.
\end{equation}
Finally, since $\left\lfloor N/\alpha \right\rfloor \le N/\alpha$, the factor $\frac{1}{\left\lfloor N/\alpha \right\rfloor}$ in the above inequality can conveniently be replaced by $\alpha/N$, which exactly gives $\Psi(N, n, \lambda ; \alpha, \beta) \le \sum_{k=-\infty}^{-1}|a_k|^2$. In view of
\eqref{estgenP} and \eqref{expverr}, this implies \eqref{lbgBp}.
\end{proof}
When applying Theorem \ref{genBp}, one may safely assume that $f$ is known,
thus also $\lambda$. Moreover, any choice of $\alpha \in (0, \alpha_0)$ determines the corresponding value $s^*(\lambda ; \alpha)$ with $\lambda < 1 < s^* < 1/\lambda$. Once $\alpha$ is chosen, one gets the best possible bound with respect to $\beta$ by taking $\beta = \beta_{\ast}(n,\lambda ; \alpha)$ in order to minimize $K(n, \lambda ; \alpha)$ as discussed earlier. In practice, one determines a decent approximation of it (for instance by sampling $K(n, \lambda ; \alpha, \cdot)$ with a finite number of values $\beta$ and retaining the one giving the minimal value).

An obvious question is
about the optimal choice of $\alpha$, however. Numerical experiments will shed some
light on this issue. Note that when $\lambda\to0$, picking
$\alpha$ arbitrary close to $(1-\lambda)/(1+\lambda)$ yields back
Theorem \ref{initp}. Though certainly not sharp,
Theorem \ref{genBp} seems the first to establish a lower bound
in terms of the respective degrees of the approximant and the approximated
Blaschke product.

\begin{remark}
In the proof of Theorem~\ref{genBp}, the identity \eqref{detind} contains a negative contribution from the non-negative Fourier modes
$-\sum_{k\ge 0}k|a_k|^2$. Discarding them as we did to obtain  \eqref{detindin}
is far from sharp when $n$ is close to $N$, and
a non-trivial lower bound on $\sum_{k\ge 0}k|a_k|^2$, uniform with respect
to $b_n\in B_n$, would  strengthen \eqref{detindin}
and thus inequality \eqref{lbgBp}. Unfortunately,  the authors could not obtain such a bound  so far without prior assumptions on the zeros of $b_n$, whereas the latter can hardly   be justified  because the candidate minimizer $b_n$ in \eqref{minHankB1}
ranges over all Blaschke products of degree $n$. Combined with the non-sharp
use of the Schwarz inequality mentioned in Remark \ref{diffSchwarz},
this structural loss due to neglecting the positive Fourier modes in \eqref{detind} accounts a good deal for the rather conservative character of the lower bound in Theorem~\ref{genBp} illustrate by the plot  in Figure \ref{secondBound}.
\end{remark}

\section{Numerical results}
\label{sec:numericalresults}
In order to estimate how pessimistic the bounds provided in the previous sections are, we used the RARL2\footnote{See \url{https://project.inria.fr/rarl2/}.} software. RARL2 is a rational approximation software designed to perform stable rational approximation (\emph{i.e.}, the equivalent of the problem RAB($n$)) for matrix-valued functions. Of course, it works in particular for matrices of size $1\times 1$, \emph{i.e.}, scalar functions which is the problem of interest for us. The problem RAB($n$) is really hard to solve in general when $n$ grows, as the criterion is a non-convex function defined on a space\footnote{As we have seen with Equation~\eqref{minp}, the best numerator $p_{n-1}$ for a given denominator $q_n$ is explicitly given by $p_{n-1} = \widetilde{q_n}{\bf P}_+(fq_n/\widetilde{q_n})$ so the problem really consists in choosing the position of the $n$ complex zeros of $q_n$ in the unit disk.} of (complex) dimension~$n$ that can have quite a lot of local minima. The software tools trying to solve this best approximation problem are usually a best-effort: they try to explore the space of parameters looking for local minima of the criterion, hoping not to forget any, and retaining the best one. RARL2 is no exception, and does so by parameterizing the space of stable rational functions with unitary matrix realizations. An interest of theoretical lower bounds such as the one discussed in the previous sections is precisely that, when they are close enough to the true minimal value, they can assess the quality of the minimizer computed by rational approximation software tools.

RARL2 can perform the optimization process from an initial point computed using Kung's truncation method~\cite{KL} or more generally any initial point provided by the user. Kung's truncation method is based on discarding the analytic part of a $L^{\infty}$ meromorphic analog of RAB($n$) (AAK\footnote{AAK stands for Adamjan-Arov-Krein.} approximation) and often provides a good initial guess for the $L^2$ problem. However, for situations under study here, it is practically useless. Indeed,  a Blaschke product being of constant modulus  on the unit circle, the constant $0$ is a best $L^\infty$ meromorphic approximant and Kung's  truncation method  yields the zero rational approximant, which is of no use to initialize the $L^2$ approximation problem under study because the poles are undefined. In such a case, RARL2 proposes an interesting alternative method to initialize the search, that consists in iteratively solving the problem RAB($n$) for increasing values of $n$ by construing local minima found at order $n$ as functions  of order $(n+1)$ with pole-zero cancellation to find initial guesses of order $n+1$. Moreover, when the function to be approximated is conjugate symmetric, RARL2 proposes an option to search for rational approximants that are also conjugate symmetric (equivalently: that have real coefficients); while a best $L^2$ rational approximant of given degree has no reason to share this property with the approximated function (indeed there are examples to the contrary),  it is sometimes desirable to get rational approximants with real coefficients and limiting the search to the latter usually results in a much faster algorithm. We will further refer the general algorithm of RARL2 as the \emph{complex} variant, and the conjugate-symmetric version as the \emph{real} variant.

The bounds presented in Theorems~\ref{initp} and~\ref{genBp} are concerned with  approximation by fractions in $\cR_{n,n}\cap H^2$, to functions in $H^2$. Therefore, from a practical point of view, in order to use RARL2, we rely on the remark after the definition of problem RAN($n$) on page~\pageref{remarkReductionFromRAtoRAB} to turn this into a problem of approximation by fractions in $\cR_{n-1,n}\cap \overline{H}^{2,0}$ to functions in $\overline{H}^{2,0}$. Notice that in the case at hand, $f$ is a Blaschke product  and thus $f^\sharp$ is $1/f$, whence the problem boils down to approximate $1/f - \overline{f(0)}$.

To illustrate Theorem~\ref{initp} we took $N=100$ and used RARL2 to estimate the best approximation of degree $n$ with $n$ from~$1$ to~$20$. Notice that, if $p_n/q_n$ is a best approximation to $z^N$ and if $\theta \in [0, 2\pi)$, $\ee^{-\text{i}N\theta}\frac{p_n(\ee^{\text{i}\theta}z)}{q_n(\ee^{\text{i}\theta}z)}$ is also a best approximation, since \[\Abs{\ee^{-\text{i}N\theta}\frac{p_n(\ee^{\text{i}\theta}z)}{q_n(\ee^{\text{i}\theta}z)} - z^N} = \Abs{\frac{p_n(\ee^{\text{i}\theta}z)}{q_n(\ee^{\text{i}\theta}z)} - (\ee^{\text{i}\theta}z)^N},\] so its $L^2$-norm is the same as the norm of $p_n/q_n-z^N$. As a consequence, the minimization has infinitely many minima which makes it very tough to solve. We managed to run the general \emph{complex} variant of RARL2 up to $n=3$ and then had to turn to the \emph{real} variant thereof  for subsequent values of $n$, in order to bring down the computational load. While we did not justify this fact rigorously, one can surmise that a best real approximant to $z^N$ is also a best complex approximant and this is indeed what our experimentations suggest for $n=1,2,3$. Since, as explained above, RARL2 is only a best-effort algorithm, it might miss the true optimum. Therefore, the only statement that can be granted is that the exact value $d_2(z^N, \cR_{n,n})$ lies in between the bound of Theorem~\ref{initp} and the approximation found by RARL2. For a comparison, we also show in Figure \ref{firstBound} the bound provided by \cite{Gloverbound}, which is easily computed in this case. Indeed, this bound is the square root of $\Sigma_{k>n}\sigma_k^2$ where $\sigma_k$ is the $k$-th singular value of $DHD$ where $D$ is an appropriate diagonal operator and $H$ the Hankel matrix of $z^{-N}$ which is very sparse. In particular
the nonzero singular values in this example are those of the principal $N\times N$
submatrix, so the computation is straightforward. It appears that this bound is here weaker than ours.

\begin{figure}[htp]
  \centering
  \includegraphics{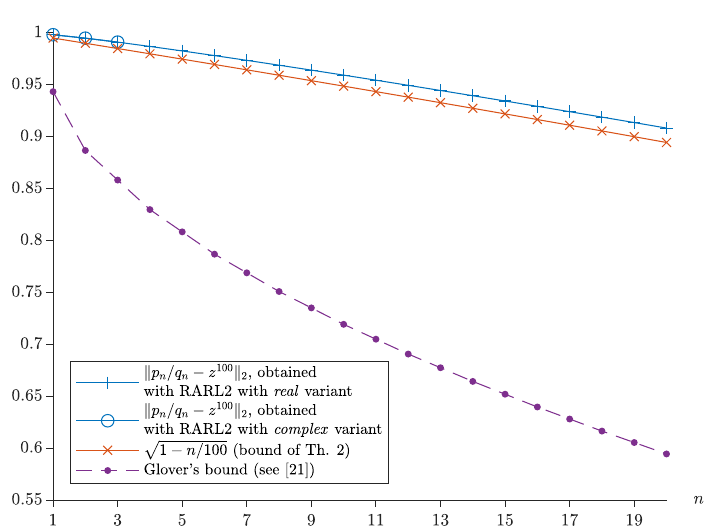}
  \caption{Comparison of the theoretical bounds of Th.~\ref{initp} and~\cite{Gloverbound} with the approximations found by RARL2 when approximating $z^N$ with $N=100$ by a fraction of $\cR_{n,n}$. The exact value $d_2(z^N, \cR_{n,n})$ is granted to lie in between RARL2's approximation and the highest theoretical bound.}
  \label{firstBound}
\end{figure}

To illustrate Theorem~\ref{genBp} we took $N=100$, $\lambda = 0.5$ and considered three different Blaschke products:
\[ B_1(z) = \left(\frac{z-\lambda}{1-\lambda z}\right)^N, \qquad B_2(z) = z^{N-1}\frac{z-\lambda}{1-\lambda z}, \qquad B_3(z) = \prod_{k=0}^{N-1} \frac{z-\lambda \ee^{\text{i}\frac{2k\pi}{N}}}{1-\lambda \ee^{-\frac{2\text{i}k\pi}{N}} z}.\] Their zero of maximal modulus has the same magnitude $\lambda$, so the bound of Theorem~\ref{genBp} is the same for the three of them. The bound provided by Theorem~\ref{genBp} much depends on $n$, $\alpha$ and $\beta$ (recall that, here, $N$ and $\lambda$ are fixed). In order to estimate the best possible bound given by Theorem~\ref{genBp}, we did the following for every value of $n \in \llbracket 1, 20\rrbracket$: sample the interval $(0, \alpha_0)$; for each value $\alpha$, numerically compute a value $\hat{\beta}(n, \lambda ; \alpha)$ such that $|\hat{\beta}(n, \lambda ; \alpha)-\beta_{\ast}(n,\lambda ; \alpha)| < 0.01$ and compute the corresponding value $\sqrt{\Psi(N,n,\lambda ; \alpha, \hat{\beta})}$. This leads a plot as in Figure~\ref{boundFunctionOfAlpha}. The zoom on the values of $\alpha$ close to its maximum clearly shows the discontinuities of the bound due to the integer part of $N/\alpha$. The plots of Figure~\ref{boundFunctionOfAlpha} are obtained for $n=4$, but one gets similar plots for other values of $n$: indeed, on this example, one observes that the value of $\alpha$ barely depends on $n \in \llbracket 1, 20\rrbracket$ and remains somewhat close to $0.31$.

\begin{figure}[htp]
  \centering
  \includegraphics[scale=0.6]{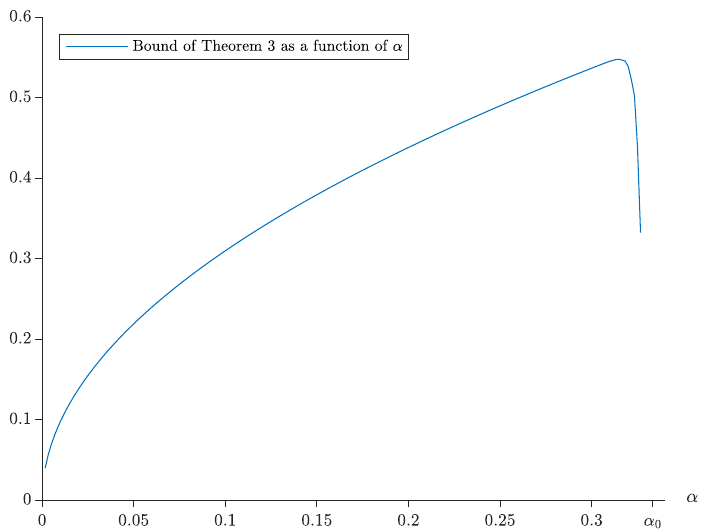}\qquad
  \includegraphics[scale=0.6]{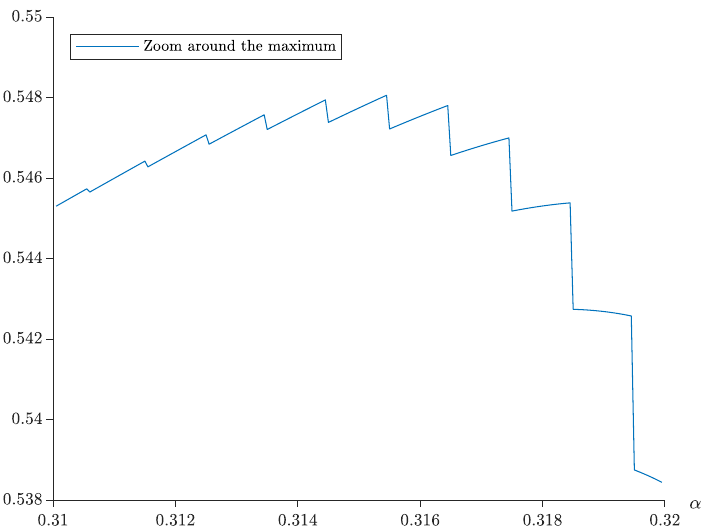}
  \caption{Values of $\sqrt{\Psi(N, n, \lambda ; \alpha, \hat{\beta})}$ plotted as a function of $\alpha$, where $\Psi$ is defined in~\eqref{def_Psi} and provides the bound of Theorem~\ref{genBp}. Here, $N=100$, $n=4$ and $\lambda = 0.5$. On the left: for $\alpha \in (0,\alpha_0)$. On the right: zoom for $\alpha$ rather close to the point where the bound is maximal.}
  \label{boundFunctionOfAlpha}
\end{figure}

\begin{figure}[htp]
  \centering
  \includegraphics{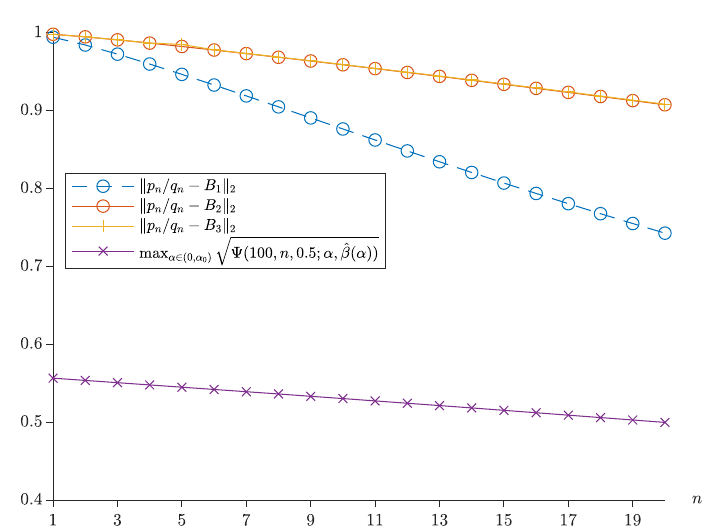}
  \caption{Comparison of the theoretical bound given in Theorem~\ref{genBp} and the approximation found by RARL2 when approximating $B_k$ ($k=1,2,3$) by a fraction of $\cR_{n,n}$, using the \emph{real} variant. In each case, the exact value $d_2(B_k, \cR_{n,n})$ is granted to lie in between the corresponding curve and the curve of the bound.}
  \label{secondBound}
\end{figure}

For each $n$, after refining the plot around $\alpha = 0.31$, we kept the maximal value $\sqrt{\Psi(N,n,\lambda ; \alpha, \hat{\beta})}$ that we could obtain. This bound, together with the results obtained with RARL2 are summed up in Figure~\ref{secondBound}. Again, we took advantage of the fact that $B_1$, $B_2$ and $B_3$ are indeed real functions on the real axis to use the \emph{real} option of RARL2 to obtain an approximation.

One can observe that the degrees  of approximation to $B_2$ and $B_3$ do not differ much from each other and are rather comparable to the one shown when approximating $z^N$. The Blaschke product $B_1$, however, with only one zero of high order away from the origin  is better approximated. In contrast with the bound of Theorem~\ref{initp} that was somehow close to sharp, there is a clear gap between the value of the bound in Theorem~\ref{genBp} and the approximation we were able to obtain; this gap is most probably due to the lack of sharpness of the bound itself rather than the poor quality of the approximants we obtained. One observes on this example where $\lambda=0.5$ and $N=100$ that we are not in a regime where it is really taken advantage of the parameter $\beta$ since $\hat{\beta}$ has been observed to be $0$ for all $n \in \llbracket 2, 20\rrbracket$ (for $n=1$ the value is $\hat{\beta} \simeq 0.5$). This is coherent with the curves shown in Figure~\ref{fig:beta-star} where the trend clearly shows that for rather moderate values of $\lambda$, the value $\beta_{\ast}(n,\lambda ; \alpha)$ quickly becomes~$0$ as $n$ grows.

\section{Appendix: Fourier coefficients of Blaschke products}
The Fourier coefficients $\{a_n\}_{n\in\NN}$ of a  Blaschke product with infinite degree
cannot be
$o(1/n)$ (little oh!) in modulus but they can be $O(1/n)$ (big oh!); this follows from
\cite[Theorems 2 \& 3]{NS}. In contrast Blaschke products of finite degree have exponentially decaying Fourier coefficients, but precise estimates seem hard to find in the literature. The purpose of this appendix is to derive such estimates, at least when the zeros are bounded away from the circle and the coefficients have sufficiently high index. This is crucial to the proof of Theorem \ref{genBp} but also of independent interest.

  Let $B\in B_n$ be a finite Blaschke product of
degree $n$, so that
\[
B(z)=\prod_{j=1}^{n}\frac{z-\lambda_{j}}{1-\overline{\lambda_{j}}z},\qquad \lambda_j\in\DD.
\]
We denote with $\widehat{B}(k)$ the $k$-th Fourier-Taylor coefficient of $B$, where $k \ge 0$, so that \[B(z) = \sum_{k=0}^{+\infty} \widehat{B}(k)z^k.\] For $w \in \mathbb{D}$, we denote with $b_w$ the elementary Blaschke factor associated with $w$:
\[ b_{w}(z)=\frac{z-w}{1-\overline{w} z}. \]
Let 
\[
\lambda:=\lambda_{{\rm max}}=\max_{j=1,\dots ,n}\abs{\lambda_{j}}<1,
\]
Let us put also 
\[
\alpha_{0}=\frac{1-\lambda}{1+\lambda}.
\]

\begin{prop}
  \label{estFTnpa}
  Let $\alpha\in(0,\alpha_{0})$ and define
\begin{equation}
  \label{defs*}
s^{*}:=\frac{\alpha^{-1}-1+(\alpha^{-1}+1)\lambda^{2}}{2\lambda\alpha^{-1}}+\sqrt{\left(\frac{\alpha^{-1}-1+(\alpha^{-1}+1)\lambda^{2}}{2\lambda\alpha^{-1}}\right)^{2}-1}.
\end{equation}
Then $s^{*}\in(1,1/\lambda)$ and one has:
\begin{equation}
  \label{majFTbn}
\abs{\widehat{B}(k)}\leq\left(\frac{b_{\lambda}(s^{*})}{s^{*}{}^{k/n}}\right)^{n}<1\qquad\mathrm{for}\quad k\geq n/\alpha. 
 \end{equation}
In particular 
\begin{equation}
  \label{majscar}
  \sum_{k\geq k_0}\abs{\widehat{B}(k)}^{2}\leq\frac{b_{\lambda}(s^{*})^{2n}}{(s^{*})^{2k_0}}\frac{\left(s^{*}\right)^2}{\left(s^{*}\right)^2-1}\qquad
  \mathrm{for}\qquad k_0\geq n/\alpha.
\end{equation}

\end{prop}

  \begin{remark}
    Similar estimates allow one to show that if ${k}\leq n \alpha$ then there exists
$s^{*}\in(\lambda,1)$ such that
\[
  \abs{\widehat{B}(k)}\leq\left(\frac{b_{-\lambda}(s^{*})}{s^{*}{}^{k/n}}\right)^{n} \leq\left(\frac{b_{-\lambda}(s^{*})}{s^{*}{}^{\alpha}}\right)^{n} < 1
\]
This estimate is not as useful to our purpose since it does not show significant decay.
\end{remark}
\begin{proof}
It is well known~\cite{Garnett} that for $z,w\in\mathbb{D}$ we have:
\begin{align*}
\frac{\abs{z}-\abs{w}}{1-\abs{w}\abs{z}}\leq\Abs{\frac{z-w}{1-\bar{w}z}}\leq\frac{\abs{z}+\abs{w}}{1+\abs{w}\abs{z}}.
\end{align*}
Indeed, if $w \in \mathbb{D}$ and $s \in (0,1)$, we have
\begin{align*}
  \max_{\abs{z} = s} \abs{b_{w}(z)} = \Abs{b_{w}\left(-\frac{w}{\abs{w}}s\right)} = \frac{s+\abs{w}}{1+\abs{w}s}.
\end{align*}
Equivalently, for $w\in\mathbb{D}$ and  $\zeta\notin\mathbb{D}$,
\begin{align*}
\frac{\abs{\zeta}+\abs{w}}{1+\abs{w}\abs{\zeta}}\leq\Abs{\frac{\zeta-\bar{w}}{1-w\zeta}}\leq\frac{\abs{\zeta}-\abs{w}}{1-\abs{w}\abs{\zeta}}.
\end{align*}
And accordingly, if $w \in \mathbb{D}$ and $s \in (1, 1/\abs{w})$, we have
\begin{align*}
  \max_{\abs{z} = s} \abs{b_{w}(z)} = \Abs{b_{w}\left(\frac{w}{\abs{w}}s\right)} = \frac{s-\abs{w}}{1-\abs{w}s}.
\end{align*}
Also observe that, for $s \in (0,1)$, the function $\lambda \mapsto \frac{s+\lambda}{1+\lambda s}$ increases for $\lambda \in (0,1)$; accordingly, if $s > 1$, the function $\lambda \mapsto \frac{s-\lambda}{1-\lambda s}$ increases for $\lambda \in (0, 1/s)$.

Now, the Fourier coefficients of $B$ can be expressed using a contour
integral:
\begin{align*}
\widehat{B}(k)=\frac{1}{2\mathrm{i}\pi}\oint_{\abs{z}=s} \frac{B(z)}{z^{k+1}}\,{\rm d}z,
\end{align*}
where $s$ can be chosen arbitrarily in $(0,1/\lambda)$ by the Cauchy theorem. Clearly, the magnitude of the integral can be bounded above so as to yield: 
\begin{align*}
  \abs{\widehat{B}(k)} \leq \max_{\abs{z}=s}{\frac{\abs{B(z)}}{{\abs{z}^{k}}}}
  = \frac{\max_{\abs{z}=s}{\abs{B(z)}}}{s^k}
  \leq \frac{1}{s^k}\,\prod_{j=1}^n \max_{\abs{z}=s} \abs{b_{\lambda_j}(z)}.
\end{align*}
From the above observations, $\max_{\abs{z}=s} \abs{b_{\lambda_i}(z)} = \max_{\abs{z}=s} \abs{b_{\abs{\lambda_i}}(z)} \le  \max_{\abs{z}=s}  \abs{b_{\lambda}(z)}$. Therefore,
\begin{align}
  \abs{\widehat{B}(k)} \leq
  \begin{cases}
    \frac{b_{\lambda}(s)^{n}}{{s^{k}}}\ \mathrm{if}\ s\in(1,1/\lambda)\\[0.1cm]
    \frac{b_{-\lambda}(s)^{n}}{{s^{k}}}\ \mathrm{if}\ s\in(0,1)
  \end{cases}.\label{eq:-1}
\end{align}

Let $a \geq \alpha^{-1}$ and define $\phi_a : s \mapsto \frac{b_\lambda(s)}{s^a}$ on $(0,1/\lambda)$. The sign of $\phi_a'(s)$ is the same as the sign of $s^2 - 2\beta_{a,\lambda}s + 1$ where $\beta_{a,\lambda} = \frac{(a+1)\lambda^2 + (a-1)}{2a\lambda}$. Since $\alpha < \alpha_0 = \frac{1-\lambda}{1+\lambda}$, it holds that $a > \frac{1+\lambda}{1-\lambda}$, whence $\beta_{a,\lambda} > 1$. Therefore $\phi_a$ increases on $(0,s_-)$, decreases on $(s_-, s_+)$, and finally increases on $(s_+, 1/\lambda)$, where \[s_{\pm} =  s_{\pm}(a,\lambda) = \beta_{a,\lambda} \pm \sqrt{\beta_{a,\lambda}^2-1}.\]
Now, for a given $\lambda$, the function $a \mapsto \beta_{a,\lambda}$ increases. Since the functions $y\mapsto y\pm\sqrt{y^{2}-1}$
are, respectively, increasing/decreasing for $y>1$, $s_{\pm}$ is, respectively, increasing/decreasing with respect to $a$. Moreover, both $s_{\pm}$ tend to $1$ when $a \to \frac{1+\lambda}{1-\lambda}$. Finally, $s_-$ tends to $\lambda$ when $a \to  +\infty$ and $s_+$ tends to $1/\lambda$ when $a \to +\infty$. This proves that, for any $\lambda$ and any $a> \alpha^{-1}  > \frac{1+\lambda}{1-\lambda}$ one has $\lambda < s_- < 1 < s^* < s_+ < 1/\lambda$ where $s^* = s_+(\alpha^{-1},\lambda)$. Since $\phi_a$ decreases strictly on $(s_-, s_+)$ and $\phi_a(1)=1$, this shows that $\phi_a(s^\star)<1$. Applying this result to $a=k/n$, we get from \eqref{eq:-1} that \eqref{majFTbn} holds, and it ensures upon summing a geometric series that \eqref{majscar} also holds. 

\end{proof}
\paragraph{Conflict of Interest:} the authors declare no
conflict of interest.

\paragraph{Funding:} no specific funding came in support of this
research.
\bibliographystyle{plain_initials}
\bibliography{biblio}
\end{document}